\documentclass{article}

\usepackage[english]{babel}

\usepackage[letterpaper,top=2cm,bottom=2cm,left=3cm,right=3cm,marginparwidth=1.75cm]{geometry}

\usepackage{amsmath}
\usepackage{graphicx}
\usepackage{amssymb}
\usepackage{amsthm}
\usepackage{braket}
\usepackage{enumerate}
\usepackage[colorlinks=true, allcolors=blue]{hyperref}
\usepackage{tikz-cd}
\usepackage{relsize}
\usepackage{keytheorems}

\usepackage{cleveref}

\usepackage{mathpazo}
\usepackage{euler}

\newcommand{\minis}{ {\mspace{1mu} }}

\newcommand{\blfootnote}[1]{%
  \begingroup
  \renewcommand\thefootnote{}\footnote{#1}%
  \addtocounter{footnote}{-1}%
  \endgroup
}

\newcommand{\subjclass}[2]{\blfootnote{\textbf{#1 Mathematics Subject Classification:}  #2}}
\newcommand{\keywords}[1]{\blfootnote{\textbf{Keywords and phrases:} #1}}

\newkeytheorem{theorem}
\newkeytheorem{ques}[name=Question]
\newkeytheorem{conj}[name=Conjecture]

\newtheorem{thm}{Theorem}[section]
\newtheorem{lemma}[thm]{Lemma}

\newtheorem{prop}[thm]{Proposition}

\newtheorem{cor}[thm]{Corollary}

\theoremstyle{definition}
\newtheorem{defn}[thm]{Definition}
\newtheorem{rmk}[thm]{Remark}
\newtheorem{ex}[thm]{Example}

\newtheoremstyle{stilecasi}
  {}% spazio sopra
  {}% spazio sotto
  {\upshape}% font del corpo (tondo/normale)
  {}% rientro
  {\bfseries}% font dell'intestazione (grassetto)
  {.}% punteggiatura dopo l'intestazione
  { }% spazio dopo l'intestazione
  {\thmname{#1}\thmnumber{ #2}\thmnote{($\minis$#3$\minis$)}}% specifica dell'intestazione

\theoremstyle{stilecasi}
\newtheorem{caso}{Case} % Senza [section], la numerazione è progressiva (1, 2, 3...)
\theoremstyle{plain}
\newtheorem{clamio}{Claim}[caso]

\crefname{thm}{Theorem}{Theorems}
\crefname{theorem}{Theorem}{Theorems}
\crefname{lemma}{Lemma}{Lemmas}
\crefname{prop}{Proposition}{Propositions}
\crefname{defn}{Defintion}{Definitions}
\crefname{rmk}{Remark}{Remarks}
\crefname{lemmadef}{Lemma-Definition}{Lemma-Definition}
\crefname{propdef}{Proposition-Definition}{Proposition-Definitions}
\crefname{ex}{Example}{Examples}
\crefname{cor}{Corollary}{Corollaries}
\crefname{ques}{Question}{Questions}
\crefname{conj}{Conjecture}{Conjectures}
\crefname{clamio}{Claim}{Claims}
\crefname{caso}{Case}{Cases}
\crefname{equation}{}{}
\Crefname{equation}{}{}
\crefformat{equation}{(#2#1#3)}

\newenvironment{dimclamio}[1][\proofname]{%
  \renewcommand{\proofname}{#1}%
  \begin{proof}[Proof of the claim.]%
}{%
  \end{proof}%
}

\numberwithin{equation}{section}

\newcommand{\affline}{\mathbb{A}}
\newcommand{\NN}{\mathbb{N}}
\newcommand{\PP}{\mathbb{P}}
\newcommand{\RR}{\mathbb{R}}
\newcommand{\CC}{\mathbb{C}}

\newcommand{\anellointeri}{{K^\circ}}
\newcommand{\massimale}{{K^{\circ\circ}}}

\newcommand{\cld}{\overline{D}}

\newcommand{\awanb}{{\mathbf{A}^1_K}}
\newcommand{\pwanb}{{\mathbf{P}^1_K}}
\newcommand{\hypb}{{\mathbf{H}_K}}
\newcommand{\bdisk}{{\mathcal{D}}}
\newcommand{\bannulus}{{\mathcal{A}}}
\newcommand{\PGL}{\operatorname{PGL}}
\newcommand{\bdv}{\operatorname{BDV}}
\newcommand{\cpa}{\operatorname{CPA}}
\newcommand{\ag}{\mathcal{AG}}
\newcommand{\supp}{\operatorname{supp}}
\newcommand{\aff}[1]{\operatorname{aff} \Set{#1}}
\newcommand{\spando}{\operatorname{span}}
\newcommand{\rf}{{\operatorname{Fix}_{>1}}}
\newcommand\dir[1]{%
    \mathchoice
        {\overrightarrow{#1}}%               1. inside \displaystyle (es. dentro \begin{equation})
        {\overrightarrow{\scriptstyle #1}}%  2. inside \textstyle (es. dentro la matematica nel testo $...$)
        {\overrightarrow{\scriptstyle #1}}%  3. inside \scriptstyle (nei pedici/apici di primo livello)
        {\overrightarrow{\scriptscriptstyle #1}}% 4. inside \scriptscriptstyle (nei pedici di secondo livello)
}
\newcommand{\meas}{\operatorname{Meas}}

\newcommand{\dep}{\operatorname{dep}}
\newcommand{\fatou}{\mathcal{F}}
\newcommand{\julia}{\mathcal{J}}

\newcommand{\fix}{{\operatorname{Fix}}}
\newcommand{\nafl}{{\operatorname{Fix}_{\geq 1}}}
\newcommand{\nas}{\operatorname{NAS}}
\newcommand{\uuno}{\mathbb{1}}
\newcommand{\id}{\operatorname{id}}
\newcommand{\uunononrepel}[1]{\uuno \left ( \substack{ #1 \text{ is a non-repelling} \\ \text{fixed direction}}\right )}

\title{Counting non-attracting subtrees\\
and the equidistribution of repelling points}
\author{Lorenzo Bottiglione}
\date{}

\begin{document}
\maketitle

\begin{abstract}
    Consider the action of a rational function on the Berkovich projective line over an algebraically closed field complete with respect to a nontrivial non-archimedean absolute value. We  say a connected component of the fixed locus is a ``non-attracting subtree'' if it contains no attracting points. We define a  ``degree of a non-attracting subtree,'' in terms of the local degrees of the rational function at its points. We show that the degrees of non-attracting subtrees sum up to the degree of the rational function. This counting result is obtained by studying the potential of a discrete measure charging each non-attracting subtree with mass equal to its degree. We prove that as we iterate the rational function these measures equidistribute to the equilibrium measure. Under certain conditions, notably in the case of polynomials, this reduces to the equidistribution of repelling points. We also show that   repelling type II points become negligible under iteration when the Lyapunov exponent is positive. These results give a partial and a complete answer to questions of Favre and Rivera-Letelier.
\end{abstract}
\subjclass{2020}{\href{https://mathscinet.ams.org/msc/msc2020.html?t=37P50}{37P50}, \href{https://mathscinet.ams.org/msc/msc2020.html?t=37P05}{37P05}.}

\section{Introduction}\label{introduzione}
    \paragraph{Background and motivation.} Let $K$ be an algebraically closed field that is complete with respect to a nontrivial non-archimedean absolute value $| \minis \cdot \minis |$ and let $R \in K(z)$ be a rational function. In \cite{MR2578470} Favre and Rivera-Letelier lay the foundations of the ergodic theory of the dynamical system given by the action of $R$ on the Berkovich projective line $\pwanb$ over $K$. The Berkovich projective line, introduced by Berkovich in \cite{MR1070709}, is a compact real tree, containing the set $\PP^1(K)$ of $K$-points (endowed with the topology induced by the absolute value $| \minis \cdot \minis |$)  as a dense subspace. Due to its good topological properties, the Berkovich projective line proved to be the right setting to study the dynamics of non-archimedean rational functions. 
    In analogy with the work of Brolin, Lyubich, and  Freire--Lopes--Ma\~n\'e in the complex case, the authors in \cite{MR2578470} prove that iterated preimages of (almost) any point under $R$ equidistribute to a canonical probability measure $\nu_R$, referred to as the equilibrium measure (see \cref{brolin} below). They also prove that periodic $K$-points equidistribute to the same measure, under certain conditions, later relaxed by Okuyama in \cite{MR4563763}. Their \textit{Question 1} then asks whether equidistribution holds for repelling periodic points. In \cite{favre2026rigiditeexpansionetentropie}, the same authors rephrase this question in a more precise way in their \textit{Conjecture 1}, which we state here below.
    \begin{conj}[\cite{favre2026rigiditeexpansionetentropie}, Conjecture 1]\label{question1prime}
        Let $R\in K(z)$ be a rational function with $\deg R \geq 2$, denote by $\rf(R^n)$ the set of repelling fixed points of $R^n$, and let
        \begin{equation*}
            \sigma_R = \sum_{\zeta \in \rf(R)} (\deg_R \zeta)\times \delta_\zeta \minis .
        \end{equation*}
        Then 
        \begin{equation*}
            (\deg R^n)^{-1}\sigma_{R^n} \xrightarrow[\ n \to \infty \ ]{} \nu_R \quad \text{weakly on } \pwanb.
        \end{equation*}
    \end{conj}
    (Here $\delta_\zeta$ is the Dirac delta  at a point $\zeta\in \pwanb$, $R^n$ denotes the $n$-fold composition of $R$ with itself --- the $n$-th \textit{iterate} of $R$ ---, and $\deg_{R^n} \zeta$ denotes the \textit{local degree} of $R^n$ at $\zeta$.)\\
    The authors show in the same paper \cite{favre2026rigiditeexpansionetentropie} that the conjecture is true  for rational functions that are \textit{affine Bernoulli} and that repelling $K$-points equidistribute for rational functions with positive \textit{Lyapunov exponent} (see \cref{thmb3frl} below).

    In other eqidistribution results concerning a sequence of finite sets in $\pwanb$ determined by iterates of $R$, the finite set determined by $R^n$ usually has some nice counting property: the preimages of a point under $R^n$ are always $\deg R^n$ if we count them according to their local degree, the periodic $K$-points of period $n$ are always $\deg R^n + 1$ if we count them as solutions to the equation $R^n(z) =z$, and the \textit{crucial measures} --- defined by Rumely in \cite{Rumely_2017} and shown to equidistribute to $\nu_R$ by Jacobs in \cite{MR3679789} and \cite{MR3761441} --- have total mass $\deg R^n -1$.
    On the other hand,  the number of fixed repelling points can vary: if $R$ has good reduction (see Section \ref{sezionedinamica} below) or is a polynomial, the number of repelling fixed points is exactly $\deg R$ (counting with multiplicity; the statement about polynomials is found in \cref{polynomials} below), while in \cref{esempiodirumely} (which is Example C in \cite{rumely2014geometry}) we have $2\deg R -2$ fixed repelling points (counting with multiplicity). This variability suggests that it might be difficult to find an object associated to $\sigma_{R^n}$ to help us study its equidistribution (such as a potential function) with a simple ``natural'' expression in terms of $R^n$. 

    In the complex case  \cref{question1prime} is true and, similarly to the non-archimedean case, and to the best of the author's knowledge, there is no known simple expression depending on $\deg R^n$ that counts the number of fixed repelling points of $R^n$. To obtain the equidistribution of repelling points for a rational function $R\in \CC(z)$ acting on $\PP^1(\CC)$, one circumvents this counting issue by first proving the equidistribution of \textit{all} periodic points (solutions to the equation $R^n(z)=z$ as $n \to \infty$) and subsequently applying the classical result of Fatou (see for example \cite{MR2193309}, Theorem 13.1)  by which non-repelling cycles are finitely many, and thus negligible with respect to the total mass $(\deg R)^n+1$ of $n$-periodic points. As we mentioned above, the solutions to the equation $R(z) =z$ equidistribute to $\nu_R$ in the non-archimedean case as well. However, these solutions are all $K$-points, while repelling points in $\pwanb$ often lie in $\pwanb \smallsetminus \PP^1(K)$, and in turn there can be infinitely many non-repelling cycles in $\PP^1(K)$, so Fatou's theorem fails. Therefore, unfortunately, what works in the complex case does not translate word by word.

    In this paper we propose a possible non-archimedean translation of the strategy outlined above to prove the equidistribution of repelling points in the complex case: we first find an equidistribution result where the counting is more regular (\cref{equidistrotau} below), and subsequently we study the error term(s) that separate it from the statement about repelling points.
        
    \paragraph{Main results.}  While fixed points in $\PP^1(K)$ are finitely many, there are often uncountably many fixed points in $\pwanb\smallsetminus \PP^1(K)$. Moreover, they are generally not isolated, and therefore behave quite differently with respect to the complex case. We therefore will focus on connected components of the fixed locus instead of individual points. Work in this direction can be found in \cite{patra2026connectedcomponentsberkovichfixed} and \cite{faber2026structurecomponentsfixedlocus}. In this paper we restrict specifically to the connected components of the locus $\nafl(R)$ of fixed points that are not attracting. We call these components \textbf{non-attracting subtrees} and denote their set by $\nas(R)$.\footnote{We avoided using ``connected component of the non-attracting fixed locus'' as terminology, because the most obvious shorthand would then have been ``non-attracting component''; in a dynamical context such a phrase would most likely be understood as referring to some kind of connected component of the Fatou set, therefore causing confusion.}
    Given a non-attracting subtree $\Gamma \in \nas(R)$ we define its \textbf{degree}
    \begin{equation}\label{gradodigamma}
        \deg_R \Gamma = 1 + \sum_{\zeta\in \Gamma} (\deg_R \zeta -1) \minis .
    \end{equation}
    Given a point $\alpha\in \pwanb\smallsetminus \PP^1(K)$, we then study the measure
    \begin{equation}\label{deftauintro}
        \tau_R^\alpha =    \sum_{\Gamma \in \nas(R)} \left ( \delta_{r_\Gamma(\alpha)} \ + \  \sum_{\zeta \in \Gamma} \left ( \deg_R \zeta - 1 \right ) \delta_\zeta  \right )  ,
    \end{equation}
    where $r_\Gamma(\alpha)$ denotes the point on $\Gamma$ that is closest to $\alpha$, \textit{the retraction of} $\alpha$ \textit{on } $\Gamma$. This measure seems to corroborate the terminology \textit{peaked component} from \cite{faber2026structurecomponentsfixedlocus}: thinking of $\tau_R^\alpha$ as our view of the non-attracting fixed locus $\nafl(R)$ from $\alpha$, of each non-attracting subtree $\Gamma$ we see its closest point $r_\Gamma(\alpha)$ and behind that we can only see the ``peaks'', that is, points $\zeta\in \Gamma$ satisfying $\deg_R \zeta \geq 2$.
    
    We prove the following.
    \begin{theorem}[label=tausuigamma]
        Let $R(z)\in K(z)$ be a nonconstant rational function. Then the positive measure $\tau^\alpha_R$ is finite, has total mass $\deg R$, and charges each $\Gamma \in \nas(R)$ with mass $\deg_R \Gamma$. In particular,
            \begin{equation*}
                \sum_{\Gamma \in \nas(R)} \deg_R \Gamma \; = \; \deg R \, .
            \end{equation*}
    \end{theorem}
    \cref{tausuigamma} is obtained by studying a potential function $T_R^\alpha : \pwanb\smallsetminus \PP^1(K) \to \RR$, whose Laplacian is shown to be $\Delta T_{R}^\alpha = \tau_R^\alpha - (\deg R)\delta_\alpha$. The potential $T_R^\alpha$ is defined by integrating a function $t_R^\alpha :\pwanb \smallsetminus \PP^1(K)\to\NN$ on segments stemming from $\alpha$ with respect to the Lebesgue measure. The function $t_R^\alpha$ can be thought of as the Berkovich analogue of a distribution function for the measure $\tau_R^\alpha$.
    In general, for a positive measure $\mu$ on $\pwanb$, we can define such a distribution function $f^\alpha_\mu : \pwanb\smallsetminus \PP^1(K) \to \RR_{\geq 0}$ by assigning to a point $\zeta\in \pwanb \smallsetminus \PP^1(K)$ the mass of $\mu$ whose only path to $\alpha$ passes through $\zeta$. For comparison, if  $\mu$ were a measure on the real line $\RR$, the role of $\alpha$ would be played by $\infty$: the value of the distribution function at $x\in \RR$ would be the mass of $\mu$ whose only path to infinity is through $x$, that is, $f^\infty_\mu(x) = \mu\{ y\in \RR \ | \ y \leq x\}$.  \\
    The function $t^\alpha_R$ though, is initially defined ``diagonally'', without mentioning the measure $\tau_R^\alpha$ directly. Namely, for $\zeta \in \pwanb \smallsetminus \PP^1(K)$, the integer $t_R^\alpha(\zeta)$ is the number of \textit{preimages under} $R$ \textit{of} $\zeta$ \textit{itself} whose only path to $\alpha$ is through $\zeta$. Using this diagonal definition, the following equidistribution result then follows easily from the aforementioned equidistribution of preimages (Brolin's \cref{brolin}).
    \begin{theorem}[label=equidistrotau]%\label{equidistrotau}
        Let $R \in K(z)$ with $\deg R \geq 2$. Then, as $n\to\infty$, the measures $\tau_{R^n}^\alpha$ equidistribute to the equilibrium measure $\nu_R$.
    \end{theorem}
    To prove \cref{question1prime} we then want to study the error term that separates $\tau_{R^n}^\alpha$ from $\sigma_{R^n}^\alpha$. Setting 
    \[\tau^\alpha_{R, \operatorname{ind}} =\sum_{\substack{\Gamma \in\nas(R) \\ : \ \Gamma \cap \rf(R) = \varnothing }} \delta_{r_\Gamma(\alpha)} \minis , \qquad 
    \tau^\alpha_{R,\operatorname{sh}} = \sum_{\substack{\Gamma \in \nas(R) \\ : \ \Gamma \cap \rf(R) \neq \varnothing}} \left (  - \minis\delta_{r_\Gamma(\alpha)} +  \sum_{\zeta\in \rf(R) \cap\Gamma} \delta_\zeta \right ),\]
    we can write the difference $\tau_R^\alpha -\sigma_R$ as the difference of its positive and negative parts:
    \begin{equation*}
        \tau_R^\alpha - \sigma_R = \tau^\alpha_{R,\operatorname{ind}} - \tau^\alpha_{R,\operatorname{sh}}.
    \end{equation*}
    The positive part $\tau^\alpha_{R,\operatorname{ind}}$ puts mass $1$ on each \textit{indifferent} subtree (i.e., a non-attracting subtree containing no repelling points). The negative part $\tau^\alpha_{R,\operatorname{sh}}$ reflects how $\tau_R^\alpha$ treats a non-attracting subtree $\Gamma$ containing multiple repelling points as a single object, in which in some sense each repelling point has ``one unit of local degree \textit{shared} with the others'' (the $-1$ and $1$ in  \cref{gradodigamma}). In analogy to the non-repelling cycles in the complex case, we could hope that the number of cycles of indifferent subtrees be finite or at least negligible with respect to the total mass $\deg R^n$ of $\tau_R^\alpha$ as $n\to \infty$, thus getting us rid of the term $\tau^\alpha_{R^n, \operatorname{ind}}$. An asymptotic study of either error term would seem to require new ideas; in this paper we limit ourselves to examining some special cases in which one or both error terms vanish for all iterates.
    
    In the case of polynomials, if we choose $\alpha$ close to infinity, both error terms vanish.
    \begin{theorem}[label=polynomials]%\label{polynomials}
        Let $R\in K[z]$ be a polynomial with $\deg R \geq 2$ and let $\alpha$ be in the Fatou component containing the point $\infty\in \pwanb$. Then $\tau^\alpha_{R^n} = \sigma_{R^n}$ for all $n$, and therefore, $\sum_{\zeta\in \rf(R)} \deg_R \zeta = \deg R$ and
        \begin{equation*}
            \left( \deg R^n \right )^{-1}\sigma_{R^n} \xrightarrow[\; n\to \infty \;]{} \nu_R \minis , \; \text{weakly on } \pwanb .
        \end{equation*}
    \end{theorem}
    When $R$ has a connected Julia set, we can make the error term $\tau^\alpha_{R^n,\operatorname{sh}}$ vanish for all $n$.
    \begin{theorem}[label=connectedcounted]%\label{connectedcounted}
         Let $R\in K(z)$ be a rational function with $\deg R \geq 2$ whose Julia set $\julia_R \subset \pwanb$ is connected. Then, $\julia_R \smallsetminus \PP^1(K) \neq \varnothing$ and taking $\alpha \in \julia_R \smallsetminus \PP^1(K)$, we have $\tau^\alpha_{R^n, \operatorname{sh}} = 0$ for all $n$.\\
         If in addition we assume that the number of indifferent subtrees of $R^n$ is $o(\deg R^n)$, or equivalently that $\sum_{\zeta\in \rf(R^n)} \deg_{R^n}\zeta  \sim \deg R^n$, then we have
        \begin{equation*}
            \left(\deg R^n \right )^{-1}\sigma_{R^n} \xrightarrow[\; n\to \infty \;]{} \nu_R \minis , \; \text{weakly on } \pwanb .
        \end{equation*}
    \end{theorem}
    In \cite{favre2026rigiditeexpansionetentropie}, Favre and Rivera-Letelier pose another two conjctures. Their \cref{conj3} involves the \textit{Lyapunov exponent}, defined as the integral $L(R)=\int_\pwanb \log || R'|| d \nu_R$, where $|| R' ||$ denotes the \textit{spherical derivative} of $R$ (see \cite{MR3369346} for the definition).
    \begin{conj}\label[conj]{conj2}
        Let $R\in K(z)$ with $\deg R \geq 2$. Then the rational function $R$ admits a unique measure of maximal entropy $\mu_R$ and we have
        \begin{equation}
            (\#\fix_{>1}(R^n))^{-1}\sum_{\zeta\in \fix_{>1}(R^n)} \delta_\zeta \ \xrightarrow[\ n\to\infty \ ]{} \ \mu_R \minis , \text{ weakly on } \pwanb .
        \end{equation}
    \end{conj} 
    \begin{conj}\label[conj]{conj3}
        Let $R\in K(z)$ with $\deg R \geq 2$. If $L(R) >0$, then we have
        \begin{equation}
            (\deg R^n)^{-1} \sum_{\zeta \in \fix_{>1} \cap \hypb} (\deg_R \zeta) \times\delta_\zeta \ \xrightarrow[\ n\to\infty \ ]{} \ 0 \minis , \text{ weakly on } \pwanb.
        \end{equation}
    \end{conj}
    The aforementioned result about the equidistribution of repelling $K$-points in the same paper \cite{favre2026rigiditeexpansionetentropie} is the following.
    \begin{thm}[\cite{favre2026rigiditeexpansionetentropie}, Th\'eor\`eme B, (3)] \label[thm]{thmb3frl}
        Let $R\in K(z)$ with $\deg R \geq 2$. If $L(R) >0$, then $\nu_R(\PP^1(K))=1$, the metric and topological entropy of $R$ are both equal to $\deg R$ and we have
        \begin{equation}
            (\deg R^n)^{-1} \sum_{\zeta\in \fix_{>1}(R^n) \cap \PP^1(K)} \delta_\zeta \ \xrightarrow[ \ n\to\infty \ ]{} \ \nu_R \minis , \text{ weakly on } \pwanb .
        \end{equation}
    \end{thm}
    The authors point out that solving \cref{conj3} affirmatively, together with the theorem above implies \cref{question1prime,conj2} in the case of positive Lyapunov exponent. We will see that using our \cref{tausuigamma} together with \cref{thmb3frl} above we can solve \cref{conj3} affirmatively. 
    \begin{theorem}[label=positivelyapunov]
        Let $R\in K(z)$ with $\deg R \geq 2$. If $L(R) >0$, then we have
        \begin{equation}
            (\deg R^n)^{-1} \sum_{\zeta \in \fix_{>1} \cap \hypb} (\deg_R \zeta) \times\delta_\zeta \ \xrightarrow[\ n\to\infty \ ]{} \ 0 \minis , \text{ weakly on } \pwanb .
        \end{equation}
    \end{theorem}
    
    \paragraph{Organization of the paper.} In Section \ref{preliminaries} we fix our notation, recall results from the literature, define some new technical notions that will play a role in the proofs to follow, and establish some basic facts about them. The discussion of these new notions is concentrated in Sections \ref{grafiquasifinitiretrazpotenz} and \ref{defqualitadinamiche}. In Section \ref{nonattcptsec} we introduce non-attracting subtrees for a fixed rational function $R\in K(z)$ and study their first basic properties (Section \ref{defefattibasesuisottalberi}), we introduce the measure $\tau_R^\alpha$ and prove \cref{tausuigamma} (Section \ref{contodeisottalberi}). In Section \ref{sezequidistro} we prove \cref{equidistrotau} (Section \ref{sezequidistrotau}) and Theorems \ref{polynomials}, \ref{connectedcounted}, and \ref{positivelyapunov} (Section \ref{sezequidistropuntirep}).
    \paragraph{Acknowledgements.} I would like to thank Farbod Shokrieh for his encouragement and support throughout the research process, and Robert L. Benedetto and Joseph H. Silverman for the helpful discussions. This work was partially supported by NSF CAREER DMS-2044564 grant and by the Jane Street Graduate Research Fellowship.
    
\section{Preliminaries} \label{preliminaries}
    We fix now for the rest of the paper a field $K$ that is algebraically closed and complete with respect to a nontrivial non-archimedean absolute value $| \minis \cdot \minis | : K \to \RR_{\geq 0}$. We denote the multiplicative subgroup of $\RR_{>0}$ that is the image of $K^\times = K \smallsetminus \{0\}$ under $| \minis \cdot \minis |$ by $\left | K^\times \right |$.  Recall that a \textbf{multiplicative seminorm} on a commutative ring with unity $A$ is a function $ | \minis \cdot \minis | : A \to \RR_{\geq 0}$ such that $|0| =0,|1| = 1$, $|a\cdot b| = |a| \cdot |b| $ and $|a+b| \leq \max \{|a|, |b|\}$, for all $a,b \in A$. We will denote by $\anellointeri = \{ a \in K \ : \ |a| \leq 1\}$ the ring of integers of $K$, by $\massimale = \{ a \in K \ : \ |a| < 1\}$ the unique maximal ideal of $\anellointeri$, and by $k = \anellointeri / \massimale$ the residue field.
    
    We include $0$ in the set $\NN$ of natural numbers. The notation $A\subset B$ will indicate non-strict containment of a set $A$ in a set $B$. We will use the symbol ``$\sqcup$'' for disjoint unions. Given a proposition $P$, the notation $\uuno(P)$ is to be evaluated as $1$ when $P$ is true and $0$ when $P$ is false. We will use a similar notation for indicator functions of subsets: if $A \subset B$, then for all $x\in B$ we define $\uuno_A(x) = 1$ if $x\in A$ and $\uuno_A(x) = 0$ if $x \in B\smallsetminus A$.  For a function $f:A\to B$ we denote by $f^{-1} C$ the preimage of a subset $C\subset B$. We denote by $\# E$ the cardinality of a set $E$. For a topological space $X$ and a subset $Y\subset X$ we denote the closure of $Y$ in $X$ by $\overline{Y}$. When considering an order $\leq$ on a topological segment, we will always assume $\leq$ to be one of the two orders inducing the topology. We will sometimes denote the finite set $\{1, \ldots, n\}$ by $[n]$.

    In the sections below we will introduce the Berkovich projective line and the action of a rational function on it. For details about the former see \cite{MR1070709}, Chapter 6 from \cite{MR3890051}, and the first two chapters from \cite{MR2599526}; for details about the latter see Chapters 7 and 8 from \cite{MR3890051} and Chapters 8 and 9 from \cite{MR2599526}.
    
\subsection{Points of the Berkovich projective line}\label{theberkoline}
    We denote by $\awanb$ the set of all multiplicative seminorms on the ring $K[z]$  that extend the absolute value on $K$ and we call it the \textbf{Berkovich affine line} over $K$. In viewing this set as a space, we will mostly denote its points by Greek letters, e.g., $\zeta \in \awanb$, writing $\zeta = | \minis \cdot \minis |_\zeta \in \awanb$ when considering $\zeta$ as a seminorm. We define the \textbf{Berkovich projective line} over $K$ as the union $\pwanb := \awanb \sqcup \{\infty\}$, where $\infty = | \minis \cdot \minis |_\infty : K[z] \to [0, \infty]$ sends a constant $a\in K$ to its absolute value $|a|$ and all other polynomials to $\infty \in [0,\infty]$. Notice that this function is not a semi-norm. We endow $\pwanb$ with the coarsest topology such that for each $f \in K[z]$ the map $\pwanb \to [0,\infty], \zeta \mapsto | \minis f \minis |_\zeta =: |  f(\zeta) |$ is continuous. This topology is often referred to as the \textit{weak} or \textit{Gel'fand} topology. It makes $\pwanb$ into a Hausdorff (sequentially) compact space. We can view the set of $K$-points of the projective line, $\PP^1(K)$, as a dense subspace of $\pwanb$, identifying $z_0 \in \PP^1(K)$ with the function $K[z]\to [0, \infty], f \mapsto |f(z_0)|$ (here we put $|\infty| = \infty$).
    
    Given $a \in K, r \geq 0$, we will define the \textbf{closed} $K$\textbf{-disk of radius} $r$ \textbf{about} $a$ as $\cld(a,r) := \{ z \in K \ |  \ |z-a| \leq r\}$. Given a decreasing sequence of disks of the above form, $S = (D_n)_{n \in \NN}$, $D_1 \supset D_2 \supset \ldots \supset D_n \supset \ldots$, setting $|f|_S = \lim_{n\to \infty} \sup_{z \in D_n} |f(z)|$ defines a multiplicative seminorm $| \minis \cdot \minis |_S$ on $K[z]$. Every  point $\zeta \in \awanb$ arises this way, and can be classified based on the sequence $(D_n)_{n \in \NN}$ as: \textbf{type I} if $D_n$ are all the same singleton; \textbf{type II} if $D_n = \cld(a,r)$ for all $n$, for some $a\in K$ and $r \in |K^\times |$; \textbf{type III} if  $D_n = \cld(a,r)$ for all $n$, for some $a\in K$ and $r \in \RR_{\geq 0 } \smallsetminus |K^\times |$; \textbf{type IV} if $\bigcap_n D_n = \varnothing$.
    
    Points of type I are exactly the elements of $\mathbb{A}^1(K) \subset \awanb$. We extend the classification to $\pwanb$ by declaring $\infty \in \pwanb$ to be of type I. We will use the notation $\hypb := \pwanb\smallsetminus \PP^1(K)$. A point $\zeta \in \pwanb$ of type I, II, or III corresponding to a disk $\cld(a,r), \ a \in K, r \geq 0$, will sometimes be denoted by $\zeta = \zeta(a,r)$.
            
    \subsection{Tree structure, segments, and the distance $\varrho( \cdot, \cdot)$}\label{strutturadalbero}
        Let $\zeta, \zeta' \in \awanb$, defined respectively by the sequences of disks $(D_n)_{n \in \NN}$ and $(D_n')_{n \in \NN}$. We write $\zeta \leq \zeta'$ if every $D_n'$ contains all but finitely many of the $D_n$, or equivalently
        \begin{equation*}
                \zeta \leq \zeta'   \quad  :\iff  \quad  \left ( \, |f(\zeta)| \leq |f(\zeta')|  , \; \forall \ f \in K[z] \, \right ).
        \end{equation*}
       This defines a partial order on $\awanb$, which we extend to $\pwanb$ by adding $\infty$ as the maximal element. Any two points $\zeta, \zeta' \in \pwanb$ then admit a join $\zeta \lor \zeta' \in \pwanb$ characterized by
        \[ \left |f(\zeta \lor \zeta') \right |  = \inf \Set{ \left |f(\xi) \right | \ | \ \xi \geq \zeta, \ \xi \geq \zeta' }, \qquad \forall \ f \in K[z].\]
        Then for $\zeta, \zeta' \in \pwanb$,  the set
        \[[\zeta, \zeta'] := \Set{ \xi \ | \  \zeta \leq \xi \leq \zeta \lor \zeta' \text{ or } \zeta' \leq \xi \leq \zeta \lor \zeta'} \]
       is the unique topological segment in $\pwanb$ having $\zeta$ and $\zeta'$ as its endpoints, which we denote by $[\zeta,\zeta']$ and refer to as a \textbf{closed segment}. The existence and uniqueness of the segment connecting two points is often referred to by saying that $\pwanb$ is \textbf{uniquely arc-connected}. Notice that when $\zeta=\zeta'$ we have $[\zeta,\zeta'] =\{\zeta\}$. We will also be using the notations $(\zeta, \zeta'):= [\zeta,\zeta'] \smallsetminus \{\zeta,\zeta'\}, (\zeta,\zeta'] := [\zeta,\zeta'] \smallsetminus \{\zeta\}, [\zeta,\zeta') := [\zeta,\zeta'] \smallsetminus \{\zeta'\}$, for respectively \textbf{open} and \textbf{half-open segments}. 
        \begin{lemma}[\cite{MR2599526}, Proposition B.1$.$ (B)] \label[lemma]{segmentsharingonept}
            Let $\zeta,\zeta',\zeta'' \in \pwanb$. If $(\zeta,\zeta') \cap (\zeta', \zeta'') = \varnothing$, then $[\zeta,\zeta'] \cup [\zeta', \zeta''] = [\zeta, \zeta'']$.
        \end{lemma}
        There exists a canonical distance function $\varrho : \hypb \times \hypb \to \RR_{\geq 0}$ (see \cite{MR2599526} \S 2.7), inducing a topology on $\hypb$ that is strictly stronger than its topology as a subspace of $\pwanb$ (where the topology on $\pwanb$ is the one described in Section \ref{theberkoline}). The two topologies coincide when restricted to segments (or finite unions thereof). For a segment $I \subset \RR$ (open, closed, or half-open), we will say a map $\gamma: I \to \hypb$ that is an isometry onto its image is an \textbf{arc-length parametrization} of the segment $\gamma(I)$. The distance $\varrho( \cdot , \cdot )$ is preserved under the action of $\PGL(2,K)$ (see \cref{sezioneazione} below).
              
    \subsection{Berkovich disks, directions at a point, and affinoids}\label{dischidirezioniaffinoidi}
	    Given $\zeta \in \pwanb$, its complement $\pwanb \smallsetminus \{\zeta\}$ can be endowed with an equivalence relation $\sim_\zeta$ defined by $\xi \sim_\zeta \xi' \iff \zeta \notin [\zeta, \zeta']$. Equivalence classes obtained this way are exactly the connected components of $\pwanb \smallsetminus \{\zeta\}$, whose set we denote by $\pi_0(\pwanb \smallsetminus \{\zeta\})$. When viewed as a subset of $\pwanb$, an element of $\pi_0(\pwanb \smallsetminus \{\zeta\})$ for some $\zeta\in \pwanb$, is called an \textbf{open } (\textbf{Berkovich}) \textbf{disk}. We can also define an equivalence relation on the set of nonempty segments $(\zeta,\xi]$ stemming from $\zeta$, calling two segments $(\zeta,\xi],(\zeta,\xi']$ equivalent if they intersect. We call such an equivalence class a \textbf{direction} at $\zeta$ and we call the set of directions at $\zeta$ the \textbf{tangent space} at $\zeta$, denoted by $T_\zeta\pwanb$.
	    \begin{prop}[\cite{MR3890051}, Proposition 6.38]\label[prop]{classifitangent}
	        \begin{enumerate}[(a)]
	            \item  The tangent space of a point of type I or type IV is a singleton.
	            \item The tangent space of a point of type III has $2$ elements.
	            \item The tangent space of a point of type II is infinite.
	        \end{enumerate}
	    \end{prop}
	    \begin{rmk}\label[rmk]{segmentiapertiduetre}
	        It follows from the proposition above that any nonempty open segment consists of points of type II and type III, and is therefore in particular contained in $\hypb$.
	    \end{rmk}
	    Given $\zeta,\xi \in \pwanb$ distinct, we will denote by $\dir{\zeta\xi}$ the direction $v\in T_\zeta \pwanb$ that is the equivalence class of the segment $(\zeta,\xi]$. If $\zeta = \xi$, the notation $\dir{\zeta\xi} = \dir{\zeta\zeta}$ will denote the point $\zeta = \xi$. Notice that in this case, $\dir{\zeta\zeta} = \zeta$ is by definition not a direction. For $\zeta \neq \xi$ we will say that $(\zeta,\xi]$ (or $(\zeta,\xi)$) is a \textbf{segment stemming into} $v$. By sending the direction $\dir{\zeta\xi}$ to the connected component in $\pi_0(\pwanb\smallsetminus \{\zeta\})$ containing $\xi$, we obtain a bijection between $T_\zeta\pwanb$ and $\pi_0(\pwanb \smallsetminus \{\zeta\})$. In particular, if $\xi,\xi' \in \bdisk(v)$ for $v \in T_\zeta\pwanb$, then $(\zeta,\xi] \cap (\zeta,\xi'] \neq \varnothing$. We will denote the open disk corresponding to a direction $v\in T_\zeta\pwanb$ by $\bdisk(v)$. The disk $\bdisk(v)$ has then boundary $\partial \bdisk(v) = \{\zeta\}$. 
	    \begin{lemma}\label[lemma]{lemmabruttu}
	        Let $\xi\in(\zeta,\zeta')\subset \pwanb$. Then 
	        \begin{equation*}
	            \overline{\bdisk(\dir{\xi\zeta'})} \subset \bdisk(\dir{\zeta\xi})
	        \end{equation*}
	    \end{lemma}
	    \begin{proof}
	         Let $\xi'\in \bdisk(\dir{\xi\zeta'})$. By our hypothesis $\dir{\xi\zeta}, \dir{\xi\zeta'}$ are distinct directions at $\xi$, whence $[\zeta,\xi)$ and $(\xi,\xi']$ are disjoint. Therefore by \cref{segmentsharingonept}  we have $(\zeta,\xi]\cup[\xi,\xi'] = (\zeta,\xi']$. This shows that $\dir{\zeta\xi} = \dir{\zeta\xi'}$ and hence that $\xi'\in \bdisk(\dir{\zeta\xi})$. Since $\xi'\in \bdisk(\dir{\xi\zeta'})$ was arbitrary and since it follows from the definitions that $\xi \in \bdisk(\dir{\zeta\xi})$, we obtain that $\overline{\bdisk(\dir{\xi\zeta'})} = \bdisk(\dir{\xi\zeta'}) \sqcup \{\xi\} \subset \bdisk(\dir{\zeta\xi})$.
	    \end{proof}
	    \begin{lemma}\label[lemma]{lemmaobbrobrioso}
	        Let $\zeta,\xi\in \pwanb$ be distinct. Then 
	        \begin{equation*}
	        	\bdisk(\dir{\zeta\xi}) \smallsetminus \bdisk(\dir{\xi\zeta}) = \pwanb \smallsetminus \bdisk(\dir{\xi\zeta}) \minis .
	        \end{equation*}
	    \end{lemma}
	    \begin{proof}
	        By the definition of disks as connected components, we have
	        \begin{equation}\label{complementuccio}
	            \pwanb \smallsetminus \bdisk(\dir{\xi\zeta}) = \{\xi\} \sqcup \bigsqcup_{u\in T_\xi\pwanb \smallsetminus \{\dir{\xi\zeta}\} }\bdisk(u) \minis .
	        \end{equation}
	        Let $u\in T_\xi\pwanb\smallsetminus \{\dir{\xi\zeta}\}$ and $\zeta' \in \bdisk(u)$. Then, since $u = \dir{\xi\zeta'} \neq \dir{\xi\zeta}$, we have that $[\zeta,\xi) \cap (\xi,\zeta'] = \varnothing$. It then follows from \cref{segmentsharingonept} that $\xi\in (\zeta,\zeta')$, and therefore by \cref{lemmabruttu},
	        \begin{equation}
	            \{\xi\} \sqcup \bdisk(u) = \overline{\bdisk(\dir{\xi\zeta'})} \subset \bdisk(\dir{\zeta\xi}) \minis .
	        \end{equation}
	        Taking a union over all $u\in T_\xi\pwanb\smallsetminus \{\dir{\xi\zeta}\}$ in the inclusion above, and looking at \cref{complementuccio}, we obtain the inclusion ``$\supset$'' in our statement. The opposite inclusion is trivial.
	    \end{proof}    
	    The intersection of finitely many open disks is called an \textbf{open affinoid}. For an open affinoid $V$ intersection of disks $\bdisk_1 , \ldots , \bdisk_r$, we can always choose the disks to not contain one another. In that case, denoting by $\zeta_i$ the boundary point of $\bdisk_i$, we will have a non-empty intersection if and only if $\bdisk_i = \bdisk(\dir{\zeta_i \zeta_j}), \ \forall \ i \neq j$, that is, if and only if each disk contains all the boundary points of the others.  It follows that for $r \geq 2$, such an affinoid is only determined by its boundary points. The notation
	    $V =: \aff{\zeta_1, \ldots , \zeta_r}$ is therefore well-defined. We denote an open affinoid with two boundary points $\zeta,\xi \in \pwanb$ by $\bannulus(\zeta,\xi)$ and call such an affinoid an \textbf{open annulus}. Notice that $\bannulus(\zeta,\xi) = \bdisk(\dir{\zeta\xi}) \cap \bdisk(\dir{\xi\zeta})$, $\overline{\bannulus(\zeta,\xi)} = \overline{\bdisk(\dir{\zeta\xi})} \cap \overline{\bdisk(\dir{\xi\zeta})}$, and $\{\zeta\} \sqcup \bannulus(\zeta,\xi) = \overline{\bdisk(\dir{\zeta\xi})} \cap \bdisk(\dir{\xi\zeta})$.
	    \begin{lemma}\label[lemma]{lemmadisgrazia}
	        Let $\zeta,\zeta',\xi \in \pwanb$ distinct. Then there exists $\zeta'' \in (\zeta,\zeta')$ such that $\xi \in \pwanb \smallsetminus \overline{\bannulus(\zeta,\zeta'')}$.
	    \end{lemma}
	    \begin{proof}
	        If $\xi \notin \overline{\bdisk(\dir{\zeta\zeta'})}$ then the conclusion holds for any $\zeta'' \in (\zeta,\zeta'')$, because $\overline{\bannulus(\zeta,\zeta')} \subset \overline{\bdisk(\dir{\zeta\zeta'})}$. Notice that such a $\zeta''$ exists, because $\zeta \neq \zeta'$, so $(\zeta,\zeta') \neq \varnothing$.
	        Suppose now instead that $\xi \in \overline{\bdisk(\dir{\zeta\zeta'})}$. Then in fact $\xi \in \bdisk(\dir{\zeta\zeta'})$ because $\xi \neq \zeta$. It follows that $\dir{\zeta\zeta'} = \dir{\zeta\xi}$, that is, there exists $\zeta'' \in(\zeta,\zeta') \cap (\zeta,\xi)$. For such a $\zeta''$ we have $\overline{\bdisk(\dir{\zeta''\zeta})} \subset \bdisk(\dir{\xi\zeta})$ by \cref{lemmabruttu}, whence $\xi \notin \overline{\bannulus(\zeta,\zeta''}) \subset \bdisk(\dir{\xi\zeta})$.
	    \end{proof}
	    Let now $\zeta\in \pwanb$ and $\xi_1, \ldots, \xi_k \in \pwanb$ with $\dir{\zeta\xi_1}, \ldots, \dir{\zeta\xi_k}$ distinct. Using the definition of disks as connected components of $\pwanb \smallsetminus \{\zeta\}$ and the identity $\bannulus(\zeta,\xi_i) = \bdisk(\dir{\zeta\xi_i}) \cap \bdisk(\dir{\xi_i\zeta})$, we can see that
	    \begin{equation}\label{affinoidesmontato}
	        \{\zeta\} \sqcup \bannulus(\zeta, \xi_1) \sqcup \ldots \sqcup \bannulus (\zeta, \xi_k) \sqcup \left ( \bigsqcup_{v \in T_\zeta\pwanb \smallsetminus \{\dir{\zeta\xi_1}, \ldots , \dir{\zeta\xi_k}\}} \bdisk(v) \right )  = \aff{\xi_1, \ldots , \xi_k} \minis . 
	    \end{equation}
    \subsection{Connected subsets of $\pwanb$ and retractions} \label{sottoinsiemiconnessi}
        The Berkovich projective line is uniquely arc connected and locally arc connected. It follows that a subset $\Gamma \subset \pwanb$ is connected if and only if it is arc-connected and that the intersection of connected subsets of $\pwanb$ is connected.
        \begin{lemma}\label[lemma]{connessocontiene}
            Let $\Gamma \subset \pwanb$ be a connected set. Let $\xi \in \Gamma, \zeta \in \overline{\Gamma}$. Then $(\xi,\zeta)\subset \Gamma$.
        \end{lemma}
        \begin{proof}
            Let $\alpha \in (\xi,\zeta)$. Since $\zeta\in \overline{\Gamma}$ and $\bdisk(\dir{\alpha\zeta})$ is a neighourhood of $\zeta$, there exists a point $\beta \in \bdisk(\dir{\alpha\zeta}) \cap \Gamma$. Since $[\xi,\alpha) \cap (\alpha,\zeta] = \varnothing$, by the bijection $T_\alpha\pwanb \to \pi_0(\pwanb\smallsetminus \{\alpha \} )$ we have that $\bdisk(\dir{\alpha\xi})$ and $\bdisk(\dir{\alpha\zeta})$ are distinct connected components of $\pwanb \smallsetminus \{\alpha \}$, and are therefore disjoint. It follows that $(\xi,\alpha) \cap (\alpha, \beta) = \varnothing$, whence by \cref{segmentsharingonept} and by $\Gamma$ being connected, we obtain that $ \alpha \in [\xi,\alpha] \cup [\alpha, \beta] = [\xi,\beta] \subset \Gamma $.
            Since $\alpha \in (\xi,\zeta)$ was arbitrary, this proves our statement.
        \end{proof}
        Let $\Gamma \subset \pwanb$ be a connected set and let $\zeta \in \Gamma$. We define the \textbf{tangent space at} $\zeta$ \textbf{in} $\Gamma$ as the set $T_\zeta\Gamma$ of directions $v \in T_\zeta\pwanb$ such that $\Gamma \cap \bdisk(v) \neq \varnothing$. We call an element $v \in T_\zeta\Gamma$ a \textbf{direction at} $\zeta$ \textbf{in} $\Gamma$. Tangent spaces at points in connected sets satisfy the following properties:
        \begin{flalign}\label{spazitangentiinternimonotoni}
            &&T_\zeta\Gamma' &\subset T_\zeta\Gamma, &&& \forall \ \zeta &\in \Gamma', && \text{for all connected sets } \Gamma' \subset \Gamma \subset \pwanb;\\
            \label{intersezionedispazitangentiinterni}
            &&T_\zeta\Gamma \cap T_\zeta\Gamma' &= T_{\zeta}(\Gamma \cap\Gamma'), &&& \forall \ \zeta &\in \Gamma\cap\Gamma', && \text{for all connected sets } \Gamma , \Gamma' \subset \pwanb.
        \end{flalign}
        We define $r_\Gamma(\zeta)$ as the closest point of $\Gamma$ to $\zeta$ in a segment $[\zeta, \xi]$ where $\xi\in \Gamma$ (this definition does not depend on $\xi$). 
        Writing $r_\Gamma:\pwanb \to \Gamma, \zeta \mapsto r_\Gamma(\zeta)$ defines a continuous map (see \cite{MR2599526}, p. 89), that is in fact a retraction, in that $r_\Gamma|_\Gamma = \id_\Gamma$. For connected sets $\Gamma \subset \Gamma' \subset \pwanb$, where $\Gamma$ is also closed, we define the retraction map $r_\Gamma|_{\Gamma'} =: r_{\Gamma,\Gamma'} : \Gamma'\to \Gamma$ by restriction.
        \begin{rmk}\label[rmk]{retrazionestessazione}
            Given a closed connected nonempty subset $\Gamma \subset \pwanb$ and $\zeta\in \pwanb \smallsetminus \Gamma$,  if $\xi\in\bdisk(\dir{r_\Gamma(\zeta) \zeta})$, then  $r_\Gamma(\zeta)=r_\Gamma(\xi)$.
        \end{rmk}

    \subsection{Measures on $\pwanb$}
	    All the measures we consider will be Borel. For a more in depth treatment of measures on $\pwanb$, see \cite{MR2599526}, \S 5.2. We denote  by $\meas(E)$ the space of (finite) signed Radon measures (differences of non-negative finite Radon measures)  on a set $E \subset \pwanb$. For $a\in \RR$, we denote by $\meas_a(E)$  the subspace of measures $\mu \in \meas(E)$ with total mass $a$, i.e., satisfying $\mu(E) = a$. We will denote by $\meas^+(E)$ the space of positive finite Radon measures on $E$ and by $\meas^+_1(E)$ the space of probability Radon measures on $E$. Finally, given $\alpha \in \hypb$, we will denote by $\meas^+_1[\alpha]$ the measures of the form $\delta_\alpha - \mu$, for $\mu \in \meas^+_1(\pwanb)$. Here $\delta_\alpha$ denotes the Dirac point mass  measure at $\alpha$.  As customary, for any signed measure $\mu$, we will write $|\mu| = \mu^+ + \mu^-$, where the pair $(\mu^+,\mu^-)$ is the Hahn-Jordan decomposition of $\mu$. For any set $E\subset \pwanb$, we will endow the vector space $\meas(E)$ --- and any of its subsets that we may consider --- with the weak-* topology (often referred to just as ``weak topology'', as we did in the introduction). Since this is the only topology we will use, we will not specify the type of convergence of measures in the sequel.
	    We will use the following version of the Portmanteau Theorem.
	    \begin{thm}[\cite{MR2599526}, Theorem A.13]\label{portmantone}
	        Let $X$ a locally compact Hausdorff space, and let $(\mu_n)_{n \in \NN}$ be a sequence of Radon probability measures, converging to a Radon probability measure $\mu$ on $X$. Then 
	        \begin{flalign*}
	            &&\mu_n(A) \ \xrightarrow[\; n\to \infty \; ]{}  \ \mu(A),  && \text{for all Borel sets } A \subset X \text{ such that } \mu( \partial A) = 0.
	        \end{flalign*}
	    \end{thm}
	    
	    In \cite{MR2599526} \S 5.2 the authors discuss a measure $\lambda$ on $\pwanb$ whose restriction to a segment coincides with the push-forward of the Lebesgue measure through an arc-length parametrization. In particular, for all $\zeta,\zeta' \in \hypb$ we have $\lambda \left ( [\zeta,\zeta' ] \right ) = \varrho(\zeta,\zeta')$. We will refer to $\lambda$ as the \textbf{Lebesgue measure}. This measure is not $\sigma$-finite on $\pwanb$. Since given two points $\zeta,\zeta' \in \pwanb$ the segment $[\zeta,\zeta']$ connecting them is unique, we will use the notation
	    \[ \int_\zeta^{\zeta'} f(\xi) d \lambda(\xi) = \int_\zeta^{\zeta'} f d \lambda := \int_{[\zeta,\zeta']} f d \lambda \]
	    for the integral of a measurable function $f : \pwanb \to \RR$ on the segment $[\zeta,\zeta']$ with respect to the Lebesgue measure. We will not take orientation into account, meaning that $\int_{\zeta'}^\zeta f d \lambda$ is the same as the integral above, rather than the opposite.\\
	    For a measurable map $f:\pwanb \to E \subset \pwanb$ the \textbf{pushforward} through $f$ of a measure $\mu$ (not necessarily Radon) is the measure $f_*\mu$ defined by $f_*\mu(F) = \mu(f^{-1}F)$ on all measurable sets $F \subset E$. Pushforward through $f$ is linear and preserves total mass, so in particular it restricts to a linear operator $\meas_0(\pwanb) \to \meas_0(E)$. We will also use the pushforward operator to avoid double subscripts, thus  preferring the notation $(r_\Gamma)_* \delta_\zeta$ over $\delta_{r_\Gamma(\zeta)}$.
	    
    \subsection{Quasi-finite  and finite subgraphs}\label{grafiquasifinitiretrazpotenz}
	    We say the set of points $\{\zeta_1, \ldots , \zeta_r\} \subset \pwanb$ \textbf{spans} $\Sigma \subset \pwanb$, and we write $\Sigma = \spando \{\zeta_1, \ldots , \zeta_r\}$, if $\Sigma$ is the smallest connected set containing $\zeta_1, \ldots , \zeta_r$; equivalently, if
	    \begin{equation}\label{spancomeunione}
	        \Sigma = \bigcup_{1 \leq i < j \leq r } [\zeta_i, \zeta_j] \, .
	    \end{equation}
	    We will call a set of the form above a \textbf{quasi-finite subgraph} (of $\pwanb$); we will call it a \textbf{finite subgraph} whenever $\zeta_1, \ldots , \zeta_r \in \hypb$. Quasi-finite subgraphs are closed. For a quasi-finite subgraph $\Sigma$, in the notations above, we have
	    \begin{flalign}\label{sigmadazeta}
	        &&\Sigma= \bigcup_{i=1}^r [\zeta,\zeta_i], &&\forall \ \zeta \in \Sigma ,
	    \end{flalign}
	    from which it follows that
	    \begin{flalign}\label{tangenteinterno}
	        &&T_\zeta\Sigma =\{\dir{\zeta\zeta_1}, \ldots, \dir{\zeta\zeta_r}\}. && \forall \ \zeta\in \Sigma .
	    \end{flalign}
	    (Here the directions $\dir{\zeta\zeta_1}, \ldots , \dir{\zeta\zeta_r}$ need not be distinct.)
	    \begin{lemma} \label[lemma]{fogliespandono}
	        Let $\Sigma \subset \pwanb$ be a quasi-finite subgraph containing more than one point. Then the set of points in $\Sigma$ that have only one direction in $\Sigma$ spans $\Sigma$.
	    \end{lemma}
	    \begin{proof}
	        Suppose $\Sigma= \spando \{\zeta_1, \ldots ,\zeta_r\}$. Suppose one of the $\zeta_i$, which up to reordering we assume to be $\zeta_r$, has two directions in $\Sigma$. Then, by \cref{tangenteinterno}, we must have $\dir{\zeta_r\zeta_i} \neq\dir{\zeta_r\zeta_j}$ for some $i\neq j$ distinct from $r$. This means that $[\zeta_i, \zeta_r) \cap (\zeta_r, \zeta_j] = \varnothing$, whence by \cref{segmentsharingonept} we have $[\zeta_i, \zeta_r] \cup [\zeta_r, \zeta_j] = [\zeta_i,\zeta_j]$. It follows that $\zeta_r \in \spando\{\zeta_1, \ldots, \zeta_{r-1}\}$, and therefore $\Sigma \subset  \spando\{\zeta_1, \ldots, \zeta_{r-1}\}$. This shows that from any set of points spanning $\Sigma$ we can always remove those that have more than one direction in $\Sigma$ and obtain a set that still spans $\Sigma$. Repeating this operation, since the set $\{\zeta_1, \ldots, \zeta_r\}$ at the beginning is finite, we will eventually obtain a set whose points all have at most one direction in $\Sigma$ and spans $\Sigma$. In fact all of these points will also have \textit{at least} one direction in $\Sigma$ because $\Sigma$ has more than one point.
	    \end{proof}
	    
	    Let $\Sigma$ be a quasi-finite subgraph and let $\zeta_1, \ldots, \zeta_r \in \Sigma$ be the points with at most one direction in $\Sigma$. We define the \textbf{interior} of $\Sigma$ as the set
	    \begin{equation*}
	        \Sigma^\circ = \Sigma \smallsetminus \{\zeta_1, \ldots , \zeta_r\}.
	    \end{equation*}
	    (Notice that this is not the topological interior of $\Sigma$ as a subspace of $\pwanb$, in either the weak or strong topology.)\\
	    If $\Sigma'\subset \Sigma$ are quasi-finite subgraphs, by \cref{spazitangentiinternimonotoni}, if a point $\zeta$ has at most one direction in $\Sigma$, it cannot have more than one in $\Sigma'$. Therefore \cref{fogliespandono} shows that
	    \begin{flalign} \label{monotoniainteriore}
	         &&(\Sigma')^\circ \subset \Sigma^\circ, && \text{ for all quasi-finite subgraphs } \Sigma' \subset \Sigma .
	    \end{flalign}
	    \begin{lemma} \label[lemma]{approxzetacirco}
	        Let $\Sigma = \spando \{\zeta_1, \ldots , \zeta_r \}$ be a quasi-finite subgraph, where $\zeta_1, \ldots, \zeta_r$ are the points with only one direction in $\Sigma$ and $r\geq 2$. Then there exist distinct points $\zeta_1^\circ, \ldots , \zeta^\circ_r \in \Sigma^\circ \subset \hypb$ satisfying the following.\\
	        If $\zeta_1', \ldots , \zeta_r'$ are such that $\zeta_i' \in [\zeta_i, \zeta_i^\circ]$, then
	        \begin{enumerate} [(a)]
	            \item for all distinct $i, j \in  \{ 1, \ldots , r\}$ we have  $[\zeta_i, \zeta_j] = [\zeta_i, \zeta_i'] \sqcup (\zeta_i', \zeta_j' ) \sqcup [\zeta_j', \zeta_j]$.
	        \end{enumerate}
	        Moreover, if we denote $\Sigma' = \spando \{\zeta_1', \ldots , \zeta_r'\}$, then:
	        \begin{enumerate} [(a)]
	            \setcounter{enumi}{1}
	            \item  for all $\zeta \in (\Sigma')^\circ$ we have $T_\zeta\Sigma' = T_\zeta\Sigma$ and each $\zeta_i'$ only has one direction in $\Sigma'$;
	            \item $\Sigma \smallsetminus (\Sigma')^\circ = \bigsqcup_{i =1}^r [\zeta_i,\zeta_i']$.
	        \end{enumerate}
	        
	    \end{lemma}
	    \begin{proof}
	        Let $\zeta_1, \ldots , \zeta_r$ be the points with only one direction in $\Sigma$. Up to replacing the $\zeta_i^\circ$ with another point in $(\zeta_i, \zeta_i^\circ)$, it is enough to prove our statement for $\zeta_i' \in [\zeta_i, \zeta_i^\circ)$, rather than $\zeta_i' \in [\zeta_i, \zeta_i^\circ]$.\\
	        For all pairs $i\neq j$, we pick $\zeta_{ij}, \zeta_{ji} \in (\zeta_i,\zeta_j)$ such that $\zeta_{ij} \neq \zeta_{ji}$ and
	        \begin{equation}\label{scomponisuzetaij}
	            (\zeta_i,\zeta_j) = (\zeta_i, \zeta_{ij}) \sqcup [\zeta_{ij}, \zeta_{ji}] \sqcup (\zeta_{ji},\zeta_j).
	        \end{equation}
	        For each $i=1, \ldots , r$, the point $\zeta_i$ only has one direction in $\Sigma$, so there exists a point $\zeta_i^\circ$ such that
	        \begin{equation}\label{sceltadizetacirco}
	            \varnothing \neq (\zeta_i, \zeta_i^\circ ] = \bigcap_{j \in \{1, \ldots , r\} \smallsetminus \{i\}}(\zeta_i , \zeta_{ij}] \, .
	        \end{equation}
	        Let $\zeta_i'\in [\zeta_i, \zeta_i^\circ )$, for $i= 1, \ldots, r$. From \cref{scomponisuzetaij} and \cref{sceltadizetacirco} we obtain that (a) holds and that the $\zeta_i^\circ$ are all distinct. Let $\Sigma'= \spando\{\zeta_1', \ldots , \zeta_r'\}$. Let $i \in \{1, \ldots, r\}$ and $\zeta\in \Sigma'\smallsetminus \{\zeta_i'\}$. By \cref{sigmadazeta}, we have that 
	        \[\Sigma' = \bigcup_{j \in \{1, \ldots , r\} \smallsetminus \{i\}} [\zeta_i',\zeta_j'] \, ,\]
	        so, $\zeta \in (\zeta_i', \zeta_j']$ for some $j \in \{1, \ldots , r\} \smallsetminus\{i\}$. From (a) it then follows that  $(\zeta_i, \zeta) = (\zeta_i, \zeta_i'] \sqcup (\zeta_i', \zeta)$. This shows that $\dir{\zeta \zeta_i} = \dir{\zeta\zeta_i'}$. Since $\zeta\in \Sigma'\smallsetminus \{\zeta_i'\}$ and $i \in \{1, \ldots, r\}$ were arbitrary, by \cref{tangenteinterno} we obtain that $T_\zeta\Sigma' = T_\zeta\Sigma$, for all $\zeta\in (\Sigma')^\circ$, which is the first part of (b).\\
	        For $i \in \{1, \ldots , r \}$ we have the following, to be justified below:
	        \begin{multline*}
	            \bigcap_{j \in \{1, \ldots, r \} \smallsetminus \{i\}} (\zeta_i', \zeta_j'] \ =  \bigcap_{j \in \{1, \ldots , r\}\smallsetminus \{i\}} \left ( (\zeta_i, \zeta_j'] \smallsetminus (\zeta_i, \zeta_i']\right ) \\
	            \supset \left ( \bigcap_{j \in \{1, \ldots , r \}\smallsetminus \{i\}} (\zeta_i, \zeta_{ji} ]\right ) \smallsetminus (\zeta_i, \zeta_i'] \ = \ (\zeta_i, \zeta_i^\circ ] \smallsetminus (\zeta_i, \zeta_i'] \ = \ (\zeta_i', \zeta_i^\circ ] \neq \varnothing .
	        \end{multline*}
	        By (a) we have $(\zeta_i, \zeta_j'] \supset (\zeta_i, \zeta_i']$ for all $i \neq j$, from which we get the first equality in the above display. For all $j \neq i$ we have $\zeta_j' \in (\zeta_j^\circ, \zeta_j] \subset (\zeta_{ji}, \zeta_j]$, whence by \cref{scomponisuzetaij}, $(\zeta_i, \zeta_j'] \supset (\zeta_i, \zeta_{ji}]$, which justifies the inclusion between the first and second line. Finally, the last two equalities follow from our definition \cref{sceltadizetacirco} and from having $\zeta_i'\in [\zeta_i, \zeta^\circ_i)$. By \cref{tangenteinterno}, this shows that $\zeta_i'$ has only one direction in $\Sigma'$, and since $i \in \{1, \ldots, r \}$ was arbitrary, this establishes the second part of (b).\\
	        Let $i \in \{1, \ldots , r\}$. Since $r \geq 2$, there exists $j \in \{1, \ldots , r\} \smallsetminus \{i\}$. By (a), the directions $\dir{\zeta_i'\zeta_i}, \dir{\zeta_i'\zeta_j'}$ are distinct; by (b), $\dir{\zeta_i'\zeta_j'}$ is the only direction at $\zeta_i'$ in $\Sigma'$, so $[\zeta_i,\zeta_i'] \cap (\Sigma')^\circ = \varnothing$. Since $i$ was arbitrary in this last equality, and since by (a) the segments $[\zeta_i,\zeta_i'], [\zeta_j', \zeta_j]$ are disjoint for $i \neq j$, we obtain the first inclusion below. The second inclusion follows from \cref{spancomeunione}, the first equality follows from (a), and the second equality follows from \cref{spancomeunione} applied to $\Sigma'$. The double inclusion proves (c).
	        \begin{equation*}
	            \begin{array}{c}
	                 \displaystyle (\Sigma')^\circ \ \subset \ \Sigma \smallsetminus \left ( \bigsqcup_{i=1}^r [\zeta_i, \zeta_i'] \right ) \ \subset \ \bigcup_{i \neq j} \left ( [\zeta_i, \zeta_j] \smallsetminus \left ([\zeta_i,\zeta_i'] \sqcup[\zeta_j,\zeta_j'] \right )\right ) \ = \ \bigcup_{i\neq j} (\zeta_i',\zeta_j') \ = \ (\Sigma')^\circ    \\
	                 \quad
	            \end{array} \qedhere
	        \end{equation*}
	        
	    \end{proof}

    \subsection{Potential theory on $\pwanb$}
	    In this section we briefly introduce the Laplacian operator and some of its properties as developed by Baker and Rumely in \cite{MR2599526}.\\
	    For a finite subgraph $\Sigma \subset \pwanb$ we denote by $\cpa(\Sigma)$ the space of continuous functions $f:\Sigma \to \RR$ that are piecewise affine, that is, for all $\zeta \in \Sigma, v \in T_\zeta\Sigma$ there exists $\zeta_v \in \Sigma \cap \bdisk(v)$ such that under an arc-length parametrization $(a,b) \to (\zeta,\zeta_v)$, the restriction $f|_{(\zeta,\zeta_v)}$ is an affine function $(a,b) \to \RR$. Baker and Rumely also define a space $\bdv(\Sigma)$ of \textbf{functions of bounded differential variation} on $\Sigma$ (whose definition will not be relevant to us) and a linear operator $\Delta_\Sigma : \bdv(\Sigma) \to \meas_0(\Sigma)$, which we call \textbf{Laplacian on} $\Sigma$. It follows from the discussion in \cite{MR2599526}, \S 3.5, that $\cpa(\Sigma) \subset \bdv(\Sigma)$. The Laplacian on $\Sigma$ is initially only defined on $\cpa(\Sigma)$ (see \cite{MR2599526}, \S 3.2), where it restricts to 
	    \begin{equation}\label{laplacianopercippiasigma}
	        \begin{array}{rcl}
	             \Delta_\Sigma : \cpa (\Sigma)  &\longrightarrow    &\operatorname{Meas}_0(\Sigma) \\
	                      f                     &\longmapsto        &  \sum_{\zeta \in \Sigma} \left ( -\sum_{v \in T_\zeta\Sigma} d_v f \right ) \delta_\zeta  \displaystyle \minis  .
	        \end{array} 
	    \end{equation}
	    Here $d_v f$ denotes the \textbf{directional derivative} of $f$ in the direction $v \in T_\zeta\Sigma$, defined as 
	    \[d_v f \ = \lim_{\substack{\xi\to \zeta \minis , \\ \zeta\in \bdisk(v)\cap \Sigma} } \frac{f(\xi) - f(\zeta) }{\varrho(\xi,\zeta)} \minis .\]
	    For a function $f: \hypb \to \RR$ we will often abuse notation and write $f\in \bdv(\Sigma)$, $f\in \cpa(\Sigma)$, and $\Delta_\Sigma f$ for, respectively, $f|_\Sigma \in \bdv(\Sigma)$, $f|_\Sigma \in \cpa(\Sigma)$, and $\Delta_\Sigma (f|_\Sigma)$. A function $f: \hypb \to \RR$ is said to be of \textbf{bounded differential variation} (on $\pwanb$) if for all finite subgraphs $\Sigma \subset \pwanb$ we have $f \in \bdv(\Sigma)$ and if there exists a constant $B(f) > 0$ such that $|\Delta_\Sigma f|(\Sigma) \leq B(f)$ for all finite subgraphs $\Sigma \subset \pwanb$. The space of functions of bounded differential variation on $\pwanb$ is denoted by $\bdv(\pwanb)$.\\
	    Baker and Rumely construct a linear operator 
	    \begin{equation*}
	        \Delta : \operatorname{BDV}(\pwanb) \longrightarrow \operatorname{Meas}_0(\pwanb)
	    \end{equation*}
	    (see \cite{MR2599526}, Chapter 5), which we call simply \textbf{Laplacian}. Its kernel is given by the constant functions. The Laplacian is related to the Laplacians of finite subgraphs via the following definition found in \cite{MR2599526} as Definition 5.15: for $f \in \bdv(\pwanb)$ there exists a unique measure $\Delta f \in \meas_0(\pwanb)$, the \textbf{Laplacian of} $f$, satisfying $\Delta_\Sigma f = (r_\Sigma)_* \Delta f$ for all finite subgraphs $\Sigma$.
	    \begin{rmk} \label[rmk]{nonservonotuttiigrafi}
	        If $\Sigma \subset \Sigma'$ are two finite subgraphs, we always have $r_\Sigma = r_{\Sigma, \Sigma'} \circ r_{\Sigma'}$, and as a consequence, $(r_\Sigma)_* = (r_{\Sigma, \Sigma'})_* \circ (r_{\Sigma'})_* : \meas_0 (\pwanb) \to \meas_0(\Sigma)$. It follows from this and from the fact that the pushforward of a measure preserves its mass, that  to determine whether $f: \hypb \to \RR$ belongs to $\bdv(\pwanb)$, it is enough to check on finite graphs $\Sigma'$ containing a given finite subgraph $\Sigma$. The same is true about checking whether $\mu = \Delta f$ for given $f\in \bdv(\pwanb), \ \mu \in \meas_0(\pwanb)$.
	    \end{rmk}
	    For $\alpha \in \hypb$ and $f \in \bdv(\pwanb)$, we say that $f$ is an \textbf{Arakelov-Green's function} relative to $\alpha$ if $f$ is upper semi-continuous on $\hypb \smallsetminus \{\alpha\}$ and $\Delta f = \delta_\alpha -\mu$ for a probability measure $\mu \in \meas^+_1(\pwanb)$. We denote the set of Arakelov-Green's functions relative to $\alpha$ by $\ag[\alpha]$. We then have the following results due to Favre and Jonsson \cite{MR2097722}.
	    \begin{prop}[\cite{MR2599526}, Proposition 8.49] \label[prop]{arachelloverdechiuso}
	        Let $\alpha \in \hypb$ and let $\{f_i\}_{i\in I}$ be a net of functions in $\ag[\alpha]$, converging pointwise to a function $f: \hypb \to [-\infty,+\infty)$ that is not identically $-\infty$. Then $f \in \ag[\alpha]$.   
	    \end{prop}
	    \begin{thm}[\cite{MR2599526}, Theorem 8.50]\label[thm]{arachelloverdesattoh}
	        Let $\alpha\in \hypb$. Equip $\ag[\alpha]$ with the topology of pointwise convergence. Then the map $-\Delta : \ag[\alpha] \to \meas^+_1[\alpha]$ is continuous and the following sequence is exact:
	        \begin{equation*}
	            0 \ \longrightarrow \  \RR \ \longrightarrow \ \ag[\alpha] \ \xrightarrow[]{ \ -\Delta \ } \  \meas_1^+[\alpha] \ \longrightarrow \ 0 \minis .
	        \end{equation*}
	    \end{thm}

    \subsection{The action of a rational function on $\pwanb$}\label{sezioneazione}
        Let $R\in K(z)$ be a rational function. We will say that $R$ is \textbf{conjugate} to some rational function $S \in K(z)$ there exists $\eta \in \PGL(2,K)$ such that $\eta \circ R \circ \eta^{-1} = S$. We will sometimes refer to a property as being \textbf{coordinate-independent} if it is preserved under conjugation.\\
        The map $R:\PP^1(K) \to \PP^1(K)$ defined by a rational function $R \in K(z)$ has a unique continuous extension $\pwanb \to \pwanb$ which we will again denote by $R$. Taking this extension respects composition: if $R,S \in K(z)$, then the map $\pwanb \to \pwanb$ induced by $R\circ S\in K(z)$ is the composition $\pwanb \xrightarrow[]{S} \pwanb \xrightarrow[]{R} \pwanb$, and the rational function $z\in K(z)$ induces the identity map on $\pwanb$. The map induced by $R$ on $\pwanb$ preserves the types of points (see Section \ref{theberkoline}) whenever $R\in K(z)$ is nonconstant. \\
        If for $a\in K$ and $r>0$ the disk $\cld(a,r)$ contains no poles of $R$, then there exist $c_n \in K$ such that the power series $\sum_{n \geq 0} c_n(z-a)^n$ converges and coincides with $R$ on $\cld(a,r)$; in this case the $\sup_{n\geq 1} |c_n| r^n$ is attained by some $|c_m|r^m$ and the map $R: \pwanb \to \pwanb$ sends the point $\zeta(a,r)$ to $\zeta(R(a), |c_m|r^m)$.

        We will denote by $\deg_R : \PP^1(K) \to \{1, \ldots , \deg R\}$ the \textbf{local degree} (or \textbf{multiplicity}) function on type I points, defined in the usual algebraic way. We can then extend $\deg_R (\cdot)$ to a function $\pwanb \to \{ 1, \ldots , \deg R \}$ by setting, for all $\zeta\in \pwanb$,  $\deg_R\zeta := \inf_U  \sup_y  \sum_{x\in U \cap R^{-1}(y)} \deg_R x$, where the infimum is taken over all open sets $U \subset \pwanb$ containing $\zeta$, and the supremum is taken over all $y \in R(U) \cap \PP^1(K)$.

        A nonconstant rational function $R \in K(z)$ induces a pullback operator
        \begin{equation*}
            R^* : \meas(\pwanb) \longrightarrow \meas(\pwanb)
        \end{equation*}
        which is defined as $R^* \delta_\zeta = \sum_{\xi \in R^{-1}(\zeta)} (\deg_R \xi)\delta_\xi$ for all $\zeta \in \pwanb$ and extended by linearity to any linear combination of Dirac deltas. (For the construction of the operator on the whole space $\meas(\pwanb)$, see for instance \cite{MR2599526}, \S 9.4.) For all $\mu \in \meas(\pwanb)$ we have the identity $R_*\mu(\pwanb) = (\deg R )  \mu(\pwanb)$.
        For $\zeta\in \pwanb$ we can also define a \textbf{local degree} function \textbf{on tangent spaces}, which we will denote the same way, $\deg_R : T_\zeta\pwanb \to \{1, \ldots, \deg_R \zeta\}$,  by setting $\deg_R v := \inf_U \sup_y \sum_{x\in U \cap R^{-1}(y)} \deg_R x$ for $v \in T_\zeta\pwanb$. Here the infimum is taken over the annuli $U = \bannulus(\zeta,\xi), \ \xi\in \bdisk(v)$, and the supremum is taken again over all $y\in R(U) \cap \PP^1(K)$.
        \begin{lemma}[\cite{MR3890051}, Theorem 7.22] \label[lemma]{expansionbydegree}
            Let $R \in K(z)$ be a nonconstant rational function, let $\zeta \in \pwanb$,
            let $v \in T_{\zeta}\pwanb$.\\
            Then there is a point $\xi \in \bdisk(v)$ such that $R$ maps $[\zeta,\xi]$ homeomorphically onto $[R(\zeta), R(\xi)]$, and
            \begin{equation*}
                \varrho(R(\xi_1),R(\xi_2)) = (\deg_R v) \times \varrho(\xi_1, \xi_2), \qquad \text{for all } \xi_1, \xi_2 \in [\zeta,\xi] \cap \hypb.
            \end{equation*}
        \end{lemma}
        When $R\in K(z)$ is a nonconstant rational function, for all $\zeta \in \pwanb$ we have an induced map 
        \begin{equation}\label{mappatangente}
            R : T_\zeta \pwanb \to T_{R(\zeta)}\pwanb
        \end{equation}
        on tangent spaces, which we still denote by $R$. In the notation of \cref{expansionbydegree}, the map above sends the equivalence class $v\in T_\zeta\pwanb$ of the segment $(\zeta,\xi]$ to the equivalence class $R(v) \in T_{R(\zeta)}\pwanb$ of the segment $(R(\zeta), R(\xi)] = R((\zeta,\xi])$.
        
        The local degree functions satisfy the following properties.
        \begin{lemma}[\cite{MR3890051}, Theorem 7.29, Theorem 7.30; \cite{MR2599526} Theorem 9.22]\label[lemma]{degreessumup}
            Let $\zeta \in \pwanb$, let $R \in K(z)$ be a nonconstant rational function, and let $w \in T_{R(\zeta)}\pwanb$. Then
            \begin{enumerate}[(a)]
                \item \label{direzionisulpunto}$\deg_R \zeta = \sum_{\substack{v \in  R^{-1}(w)}} \deg_R(v)$;
                \item $\deg R = \sum_{\xi \in R^{-1}(\zeta)} \deg_R \xi \, $;
                \item for all $v\in T_\zeta\pwanb$ there exists a segment $(\zeta,\zeta')$ stemming into $v$ on which $\deg_R : \pwanb \to \{ 1, \ldots , \deg R \}$ is constantly equal to $\deg_R v$.
            \end{enumerate}
        \end{lemma}
        \begin{lemma}[\cite{MR3890051}, Theorem 7.34]\label[lemma]{imgofdisk}
            Let $v \in T_\zeta \pwanb$ for some $\zeta \in \pwanb$ and $R \in K(z)$ be a nonconstant rational function. Then $R(\bdisk(v))$ is either the disk $\bdisk(R(v))$, which has boundary point $R(\zeta)$, or otherwise $R(\bdisk(v)) = \pwanb$.
        \end{lemma}
        
        The following lemma is a consequence of \cref{expansionbydegree} and \cref{degreessumup} (a), (b). We omit an explicit proof.
        \begin{lemma}[Preimage of a segment short enough] \label[lemma]{preimageofsegment}
            Let  $R(z) \in K(z)$ be a nonconstant rational function and $I \subset \pwanb$ be a segment parametrized by $[0,L] \to I, r \mapsto \zeta_r$, for some $L \in (0, \infty]$. Let  $m \in \NN \smallsetminus \{0\}$ be the number of distinct preimages of $\zeta_0$. Denote these preimages by $\xi_i \in \pwanb, \ i \in [m]$.\\
            Then there exist $\varepsilon>0$, integers $\ell_i \geq 1$ for each $i \in [m]$, and disjoint segments $(\xi_i, \xi_i^k(\varepsilon)] \subset \pwanb$ for $k \in [\ell_i]$ and $i \in [m]$, with parametrizations
            \[[0,\varepsilon] \longrightarrow [\xi_i, \xi_i^k(\varepsilon)], \quad r \longmapsto \xi_i^k(r), \ \xi_i^k(0) = \xi_i \, , \qquad  k \in [\ell_i], \  i \in [m],  \]
            such that
            \[R^{-1}(\zeta_r) = \Set{\xi_i^k(r) \ | \ i \in [m], k \in [\ell_i] }, \qquad 0 \leq r \leq \varepsilon  , \]
            and such that $R$ maps $(\xi_i, \xi_i^k(\varepsilon)) \longrightarrow (\zeta_0, \zeta_{\varepsilon})$ bijectively and with degree constantly equal to \\
            $\deg_R {  \overrightarrow{ \scriptstyle \xi_i \xi_i^k(\varepsilon)}}$, for all $k \in [\ell_i], i \in [m]$.
        \end{lemma}
        
    \subsection{Fixed points and fixed directions} \label{defqualitadinamiche}
        Let $R \in K(z)$. We denote by $\fix(R)$ the \textbf{fixed locus} of $R$, i.e., the set of points in $\pwanb$ fixed by $R$. Let $\zeta \in \fix(R) \cap \awanb$. If $\zeta \in \affline^1(K)$, we have a power series expansion $R(z) = \sum_{n\geq 0} c_n (z-\zeta)^n, \ c_n \in K$, converging in a disk $\cld(\zeta,r)$ for some $r>0$.   We define the \textbf{multiplier} $\lambda$ of $\zeta$ as
        \begin{equation*}
            \lambda = c_1 \in K, \ \  \text{if} \ \zeta \in \affline^1(K), \qquad  \lambda = \deg_R \zeta \in \NN \smallsetminus \{0\}, \ \ \text{if} \ \zeta \in \hypb.
        \end{equation*}
        Multipliers are preserved under conjugation of $R$ by elements of $\PGL(2,K)$. Therefore, if $\zeta = \infty$ is fixed by $R$, we define its multiplier as the multiplier of some $a \in \affline(K)$ for a rational function $\eta \circ R \circ \eta^{-1}$, where $\eta \in \PGL(2,K)$ sends $\infty$ to $a$.\\
        We then say that a fixed point $\zeta \in \pwanb$ with multiplier $\lambda$ is \textbf{attracting} if $|\lambda| < 1$, \textbf{indifferent} if $|\lambda| = 1$, and repelling if $|\lambda| > 1$. We denote the sets of attracting, indifferent, and repelling fixed points by, respectively, $\fix_{<1}(R), \fix_{=1}(R), \fix_{>1}(R)$. It follows that $\fix(R) = \fix_{<1}(R) \sqcup \fix_{=1}(R)\sqcup \fix_{>1}(R)$.\\
        The following proposition collects some facts relating the Berkovich classification of points to the classification of fixed points introduced above. Points (a) and (c) can be found in \cite{MR3890051} \S 8.2, point (b) follows from (c), and point (d) follows from the definitions by expanding $R$ as a power series.
        \begin{prop}\label[prop]{classifidinamica}
            Let $R\in K(z)$ and let $\zeta\in \pwanb$ be fixed by $R$.
            \begin{enumerate}[(a)]
                \item If $\zeta$ is attracting, then it is type I.
                \item If $\zeta$ is repelling, then it is either type I or type II.
                \item If $\zeta$ is type III or IV, then it is indifferent, so in particular $\deg_R\zeta = 1$.
                \item If $\zeta$ is a fixed type I point and $\deg_R \zeta \geq 2$, then $\zeta$ is attracting.
            \end{enumerate}
        \end{prop}        
        Let now $R \in K(z)$ be a nonconstant rational function and let $\zeta\in \pwanb$. We say that $v\in T_\zeta\pwanb$ is a \textbf{fixed direction} if the map on tangent spaces maps $v$ to itself. Notice that this implies that $\zeta$ is a fixed point. Given a fixed direction $v = R(v) \in T_\zeta\pwanb$, we describe the direction $v$ with the adjective in the left column if there exists a $\xi \in \bdisk(v)$ such that for all $\xi' \in (\zeta,\xi)$ the corresponding condition in the right column is satisfied:
        \begin{equation*}
            \begin{array}{ll}
                                    \textbf{attracting},       \qquad &    R((\zeta,\xi']) \subset (\zeta,\xi') \minis ;\\
                                    \textbf{non-attracting},       \qquad &    R((\zeta,\xi']) \supset (\zeta,\xi'] \minis ;\\
                                    \textbf{non-repelling},       \qquad &    R((\zeta,\xi']) \subset (\zeta,\xi'] \minis ; \\
                                    \textbf{repelling},       \qquad &    R((\zeta,\xi')) \supset (\zeta,\xi'] \minis ;\\
                                    \textbf{indifferent},       \qquad &    R((\zeta,\xi']) = (\zeta,\xi'] \minis .
            \end{array}
        \end{equation*}
        Since for $\eta \in \PGL(2,K)$ the map $\eta :\pwanb\to \pwanb$ is a homeomorphism, these definitions are coordinate-independent.
        \begin{rmk}\label[rmk]{indifferentdirections}
            If $\zeta\in \pwanb$ is fixed and $v\in T_\zeta \pwanb$, then 
            \begin{enumerate}[(a)]
                \item $v$ is an indifferent fixed direction if and only if $R$ restricts to the identity on some segment stemming from $\zeta$ into $v$;
                \item in particular, if $v$ is an indifferent fixed direction, then by \cref{expansionbydegree} we have $\deg_R v = 1$;
                \item  if $\zeta \in \hypb$ and $v$ is fixed, then  by \cref{expansionbydegree} the fixed direction $v$ is non-repelling if and only if $\deg_R v = 1$.
            \end{enumerate} 
        \end{rmk}
        It is a priori not obvious that a fixed direction $v$ at a fixed point $\zeta \in \pwanb$ must be either repelling or non-repelling, or that it must be either attracting or non-attracting. The following proposition gives us the dichotomies and trichotomy that we expect.
        \begin{prop}\label[prop]{dicotomiadirezioni}
            Suppose a nonconstant rational function $R \in K(z)$ fixes a point $\zeta \in \pwanb$ and a direction $v \in T_\zeta\pwanb$. Then:
            \begin{enumerate}[(a)]
                \item $v$ is either repelling or non-repelling;
                \item $v$ is either attracting or non-attracting;
                \item $v$ is either attracting, indifferent, or repelling.
            \end{enumerate}
        \end{prop}
        \begin{proof}
            (a). Suppose that $v$ is not repelling. We will show that it is non-repelling.\\
            By \cref{expansionbydegree} there exists $\xi \in \bdisk(v)$ such that on $[\zeta,\xi]$ the restriction of $R$ is a homeomorphism onto its image. We will assume without loss of generality that $\xi \in \hypb$.
            By the definition of map on tangent spaces \cref{mappatangente}, since $v$ is fixed, we have that $(\zeta,\xi] \cap (R(\zeta),R(\xi)]$ is a non-empty segment, one of whose endpoints is $\zeta$. Up to replacing $(\zeta,\xi]$ with this intersection, we then have that $R$ maps $[\zeta,\xi]$ into some segment $[\zeta,\beta]$ such that $[\zeta,\xi] \subset [\zeta,\beta]$.  Consider the order $\leq$ inducing the topology on the segment $[\zeta,\beta]$, where we take $\zeta$ to be the minimal element. Since $v$ is not a repelling direction, $R(\xi') > \xi'$ cannot hold for all $\xi' \in (\zeta,\xi)$. So, up to replacing $\xi$ with a point in $(\zeta,\xi]$ for which the aforementioned inequality fails, we can assume $R(\xi) \leq \xi$. Now, since $R|_{[\zeta,\xi]}$ is a homeomorphism onto its image and fixes $\zeta$, it must preserve the order $\leq$. Let $\xi' \in (\zeta,\xi)$. Then,
            \[ R(\xi') < R(\xi) \leq \xi \, .\]
            Since $\varrho(R(\xi'), R(\xi)) \geq \varrho(\xi',\xi)$ by \cref{expansionbydegree}, we must have $\xi' \geq R(\xi')$. By the arbitrariness of $\xi' \in (\zeta,\xi)$, this shows in fact that $R((\zeta,\xi']) \subset (\zeta,\xi']$ for all $\xi'\in (\zeta,\xi)$, that is, $v$ is non-repelling. The proof of point (b) is analogous.

            Notice that it is immediate from the definitions that, if $v$ is a fixed direction, it being an indifferent fixed direction is equivalent to it being both a non-repelling and non-attracting fixed direction. Point (c) then follows from (a) and (b). 
        \end{proof}
        \begin{prop}   \label[prop]{taledirezionetalepunto}
            Let $\zeta \in \pwanb$ be fixed by a nonconstant rational function $R \in K(z)$.
            \begin{enumerate}[(a)]
                \item The point $\zeta$ has an  attracting fixed direction if and only if $\zeta$ is an attracting fixed point.
                \item If $\zeta$ has a repelling fixed direction, then $\zeta$ is a repelling fixed  point. If $\zeta$ is type I, then the converse holds.
            \end{enumerate}
        \end{prop}
        \begin{proof}   
            (a). We first show that 
            \begin{equation}\label{direziattraquinditipouno}
                \text{if } \zeta \text{ has an attracting fixed direction, then } \zeta \text{ must be type I.}
            \end{equation}
            Let $v \in T_\zeta \pwanb$ be a attracting fixed direction and let $\xi \in \bdisk(v)$ be such that 
            \begin{flalign}\label{attrattino}
                &&R([\zeta,\xi']) \subset [\zeta,\xi'), &&\forall \ \xi' \in (\zeta, \xi).
            \end{flalign}    
            Then by \cref{expansionbydegree}, up to replacing $\xi$ with another point in $(\zeta,\xi)$, we have $\varrho(R(\xi_1), R(\xi_2)) \geq \varrho(\xi_1,\xi_2)$ for all $\xi_1,\xi_2 \in [\zeta,\xi] \cap \hypb$. It follows from this and from  \cref{attrattino} that $\zeta \in \pwanb \smallsetminus \hypb = \PP^1(K)$, that is, $\zeta$ is a type I point.

            Let now $\zeta \in \pwanb$ be a type I fixed point. Up to change of coordinates, we can suppose $\zeta=0$. Then as explained in Section \ref{sezioneazione}, we have a power series expansion $\sum_{n\geq 1} c_n z^n$, converging to $R(z)$ for $|z| \leq r$, for some $r>0$. Denote by $k$ the smallest integer for which $c_k \neq 0$. The $\sup_{n\geq 1} |c_n|r^n$ is attained, so that up to taking a smaller $r>0$ we can assume that $\sup_{n\geq 1} |c_n|s^n = |c_k|s^k$ for all $0< s \leq r$. The restriction of $R: \pwanb \to \pwanb$ to the segment $[0,\zeta(0,r)]$ is then given by
            \begin{equation*}
                \zeta(0,s) \ \longmapsto \ \zeta(0,|c_k| s^k).
            \end{equation*}
            Using the order on $\pwanb$ defined in Section \ref{strutturadalbero} we then have that
            \[ \big  (  \ R \left ( \zeta(0,s)\right ) < \zeta(0,s) \ \text{ for all } s>0 \ \text{small enough} \ \big ) \ \iff \ \big ( \ k \geq 2 \ \ \text{or} \ \  0 < |c_1| < 1 \ \big ).\]
            The condition on the left is equivalent to  the only direction at $0$ being attracting and the condition on the right is equivalent to $|c_1|<1$, that is, $0$ being attracting. We have thus proven point (a).

            (b). If $\zeta$ is type I, then the proof is analogous to our reasoning for point (a) after we finished establishing \cref{direziattraquinditipouno}. Suppose instead $\zeta \in \hypb$.  Let $v \in T_\zeta\pwanb$ be a repelling fixed direction and let $\xi \in \bdisk(v)$ be such that
            \begin{flalign*}
                &&R([\zeta,\xi')) \supset [\zeta,\xi'], &&\forall \ \xi' \in (\zeta,\xi) . 
            \end{flalign*}
            Then by \cref{expansionbydegree} we must have $\deg_R v \geq 2$, whence by \cref{degreessumup} (b) we conclude that $\deg_R \zeta \geq 2$ as well, so $\zeta$ is repelling.
        \end{proof}
        
    \subsection{Dynamics of a rational function} \label{sezionedinamica}
    	Let $R \in K(z)$ be a rational function. Recall that we denote the $n$-th iterate of $R$ by $R^n$, for $n \in \NN \smallsetminus \{0\}$. In addition we define $R^0(z) := z \in K(z)$, which induces the identity $R^0 = \id_\pwanb :\pwanb \to \pwanb$ on the Berkovich projective line. Given a set $E \subset \pwanb$, we will denote its iterated preimage as $R^{-n}E := (R^n)^{-1}E$.
        
        We say a point $z \in \PP^1(K)$ is in the \textbf{exceptional set} of $R$ if its backwards orbit $\bigcup_{n \geq 1} R^{-n}(z)$ is finite.  The exceptional set consists of at most two points, unless $\operatorname{char} K >0$ and $R$ is conjugate to an iterate of the Frobenius automorphism, in which case the exceptional set is countably infinite.
        \begin{thm}[\cite{MR2578470}, Th\'eor\`eme A] \label{brolin}
             Let $R\in K(z)$ be a rational function with $\deg R\geq 2$. Then there exists a Radon probability measure $\nu_R \in \meas(\pwanb)$ such that for all probability measures $\mu \in \meas(\pwanb)$ we have
             \begin{equation*}
                 (\deg R)^{-n}(R^n)^* \mu \; \xrightarrow[ \; n\to\infty  \; ]{} \; \nu_R
             \end{equation*}
             if and only if $\mu$ does not charge the exceptional set of $R$. In particular, $\nu_R$ is the unique probability measure not charging the exceptional set and satisfying $R^*\nu_R = (\deg R) \nu_R$.
        \end{thm} 
        We define the \textbf{Julia set} of a rational function $R \in K(z)$ with $\deg R \geq 2$ as the set $\julia_R := \supp \nu_R$. Since $\nu_R$ is a Radon measure, the Julia set is closed. Because $R^*\nu_R = (\deg R) \nu_R$, the Julia set is strongly invariant, i.e., $R^{-1}\julia_R = \julia_R$. The complement of the Julia set is called the \textbf{Fatou set} and we will denote it by $\fatou_R$. We refer to its connected components simply as \textbf{Fatou components}.

        \begin{prop}[\cite{MR3890051}, Theorem 8.7] \label[prop]{chistadove}
            Let $R\in K(z)$ have $\deg R\geq 2$ and let $\zeta \in \pwanb$ be fixed. Then
            \begin{enumerate}[(a)]
                \item if $\zeta$ is type III or type IV, then $\zeta \in \fatou_R$;
                \item if $\zeta$ is repelling, then $\zeta\in \julia_R$;
                \item if $\zeta$ is type II and indifferent, then
                \begin{equation*}
                    \zeta \in \julia_R \quad \iff \quad \left ( \exists \ v \in T_\zeta\pwanb \quad  : \quad \left ( R(\bdisk(v)) = \pwanb \right ) \ \land  \ \left ( R^n(v) \neq v, \ \forall \ n \geq 1 \right )\right )
                \end{equation*}
            \end{enumerate}
        \end{prop}
        Let $R \in K(z)$ and let $R(z) = P(z)/Q(z)$, for $P,Q \in \anellointeri[z]$ polynomials, with at least one of them monic. (Notice that such a choice of $P,Q$ is always possible). We say $R$ has \textbf{good reduction} if $\deg (\overline{P}/\overline{Q}) = \deg R$, where $\overline{P}, \overline{Q}$ are the images of $P,Q$ under the map $\anellointeri[z] \twoheadrightarrow k[z]$ induced by the quotient map $\anellointeri \twoheadrightarrow \anellointeri/\massimale = k$. This notion is independent of our choice of $P,Q$. We say $R$ has \textbf{potential good reduction} if it is conjugate to a rational function with good reduction. 
        \begin{thm}[\cite{MR2578470}, Th\'eor\`eme E] \label{reduciergodo}
            Let $R \in K(z)$ be a rational function with $\deg R \geq 2$. Then the following are equivalent.
            \begin{enumerate}[(a)]
                \item $R$ has potential good reduction.
                \item The measure $\nu_R$ charges a point.
                \item The measure $\nu_R$ equals $\delta_\zeta$ for some type II point $\zeta \in \pwanb$.
            \end{enumerate}
        \end{thm}

\section{Non-attracting subtrees and their degrees} \label{nonattcptsec}
    \subsection{Definitions and basic properties}\label{defefattibasesuisottalberi}
	    Let $R \in K(z)$ be a nonconstant rational function. We say a point $\zeta \in \pwanb$ is a \textbf{non-attracting fixed point} if it is fixed by $R$ and it is not attracting.  We define the  \textbf{non-attracting fixed locus} as 
	    \[\nafl(R) := \Set{ \zeta \in \pwanb \ | \ \zeta \text{ is a non-attracting fixed point}}.\]
	    \begin{rmk} \label[rmk]{gradiinnafl}
	        Let $\zeta\in \nafl(R)$. It follows from \cref{classifidinamica} that:
	        \begin{enumerate}[(a)]
	            \item  if $\zeta$ is not repelling, then $\deg_R\zeta = 1$;
	            \item $\deg_R \zeta \geq 2$ if and only if $\zeta$ is a type II repelling fixed point.
	        \end{enumerate}
	    \end{rmk}
	
	    \begin{lemma} \label[lemma]{naflnonvuoto}
	            The non-attracting fixed locus $\nafl(R)$ is nonempty.
	    \end{lemma}
	    \begin{proof}
	        If $\deg R \geq 2$, then $R$ has a repelling fixed point by \cite{MR3890051} Theorem 12.5. For the case $\deg R = 1$, we will use the discussion in \cite{MR3890051}, \S 1.5. If $R$ only has one fixed point in $\PP^1(K)$, then that point has multiplier $1$, and is therefore non-attracting. If $R$ has two distinct fixed points in $\PP^1(K)$, then $R$ is conjugate to a linear transformation $\eta\in \PGL(2,K)$ of the form $\eta(z) = \lambda z $ for some $\lambda \in K^\times$. Such an $\eta$ has fixed points $0,\infty$, with multipliers $\lambda,\lambda^{-1}$, respectively. So $0$ and $\infty$ cannot both be attracting.
	    \end{proof}
	    We will refer to a connected component of the non-attracting fixed locus as a \textbf{maximal non-attracting fixed subtree}, or just \textbf{non-attracting subtree} for short, and we will denote the set of such components by $\nas(R) := \pi_0(\nafl(R))$. We say that $\Gamma\in \nas(R)$ is a (\textbf{maximal}) \textbf{repelling} (\textbf{fixed}) \textbf{subtree} if it contains a fixed repelling point; otherwise we say $\Gamma$ is an (\textbf{maximal}) \textbf{indifferent} (\textbf{fixed}) \textbf{subtree}.
	    \begin{lemma}\label[lemma]{gammachiuso}
	        Non-attracting subtrees are closed.
	    \end{lemma}
	    \begin{proof}
	        Let $\zeta \in \overline{\Gamma}$. By continuity, $R(\zeta) = \zeta$ is fixed. We then only need to prove that $\zeta$ is non-attracting.
	         We can reduce to the case where $\Gamma$ has more than one point. We then have a $\xi \in \Gamma \smallsetminus \{\zeta\}$, and by \cref{connessocontiene} we have $(\zeta,\xi) \subset \Gamma$, so $R|_{(\zeta,\xi)} = \id_{(\zeta,\xi)}$, and therefore $\dir{\zeta\xi} \in T_\zeta\pwanb$ is a non-attracting fixed direction. It then follows from \cref{taledirezionetalepunto} (a) and \cref{classifidinamica} (a) that $\zeta$ cannot be attracting.
	    \end{proof}
	    \begin{lemma}\label[lemma]{indifferelligammelli}
	        Let $\zeta\in \Gamma \in \nas(R)$ and $v \in T_\zeta\pwanb$. Then $v$ is an indifferent fixed direction if and only if $v \in T_\zeta\Gamma$.
	    \end{lemma}
	    \begin{proof}
	        Suppose there is a segment $(\zeta,\xi)$ stemming into $v$ on which $R$ restricts to the identity. Then all points in $(\zeta,\xi)$ are fixed and have more than one direction, whence by \cref{classifitangent} (a) and \cref{classifidinamica} (a), we have $(\zeta,\xi) \subset \nafl(R)$, and therefore in fact $[\zeta,\xi) \subset \Gamma$. So $v \in T_\zeta\Gamma$. Vice versa, if $v\in T_\zeta\Gamma$, any segment stemming into $v$ will contain a point $\xi \in \Gamma$, whence $(\zeta,\xi) \subset \Gamma$ and $R|_{(\zeta,\xi)} = \id_{(\zeta,\xi)}$. Therefore, considering \cref{indifferentdirections}, we have shown that
	        \[ v \in T_\zeta\Gamma \ \iff \  \exists \ \xi \in \bdisk(v) \ : \ R|_{(\zeta,\xi)} = \id_{(\zeta,\xi)} \ \iff \ v \text{ is an indifferent fixed direction.} \qedhere \]
	    \end{proof}
	    
	    \begin{defn}
	        For $\Gamma \in \nas (R)$, we define the \textbf{degree} of $\Gamma$ with respect to $R$ as 
	        \begin{equation}\label{defgradociocco}
	            \deg_R \Gamma := 1 + \sum_{\zeta \in \Gamma} (\deg_R \zeta - 1 ) \, .
	        \end{equation}
	    \end{defn}
	    \begin{rmk}
	            The degree of a non-attracting component is a well-defined positive integer, as all  but finitely many summands in the sum above are zero: all points in $\Gamma$ are fixed, and those satisfying $\deg_R \zeta \geq 2$ are finitely many.  The latter finiteness is found for example in \cite{Rumely_2017}, Corollary 6.2.
	    \end{rmk}
	    \begin{lemma}\label[lemma]{piuvicinoadalphaspento}
	        Let $\alpha \in \pwanb$ and let $\Gamma\in \nas(R)$. Then
	        
	        \begin{enumerate}[(a)]
	            \item $\uunononrepel{\dir{r_\Gamma(\alpha)\alpha}} = 0 \minis .$
	            \item $\uunononrepel{\dir{\zeta\alpha}} = 1, \quad \forall \ \zeta \in \Gamma \smallsetminus \{r_\Gamma(\alpha)\} \minis$.
	        \end{enumerate}
	    \end{lemma}
	    \begin{proof}
	        (a). If $r_\Gamma(\alpha) = \alpha$, then $\dir{r_\Gamma(\alpha)\alpha} = \dir{\alpha\alpha} = \alpha$, which is not a direction, so  (a) holds. Suppose now instead that $r_\Gamma(\alpha) \neq \alpha$ and that the direction $\dir{r_\Gamma(\alpha)\alpha}$ is fixed. Then by the definition of $r_\Gamma : \pwanb \to \Gamma$ we have  $\dir{r_\Gamma(\alpha) \alpha} \notin T_{r_\Gamma(\alpha)} \Gamma$. It follows from \cref{indifferelligammelli} that $\dir{r_\Gamma(\alpha)\alpha}$ is not an indifferent fixed direction. Since $r_\Gamma(\alpha) \in \Gamma \subset \nafl(R)$ is non-attracting, by \cref{taledirezionetalepunto} (a), the direction $\dir{r_\Gamma(\alpha)\alpha}$ cannot be an attracting fixed direction either, whence we conclude by \cref{dicotomiadirezioni} that it must be repelling.
	        
	        (b). Suppose now that $\zeta \in \Gamma \smallsetminus \{r_\Gamma(\alpha)\}$. Then by the definition of $r_\Gamma : \pwanb \to \Gamma$ we have $\zeta \neq \alpha$, so $\dir{\zeta\alpha}$ is a direction. Moreover, again by the definition of $r_\Gamma : \pwanb \to \Gamma$, we have $\varnothing \neq (\zeta,r_\Gamma(\alpha)) \subset \Gamma \cap (\zeta,\alpha)$, whence $R|_{(\zeta,r_\Gamma(\alpha))} = \id_{(\zeta,r_\Gamma(\alpha))}$, so that $\dir{\zeta r_\Gamma(\alpha)} = \dir{\zeta\alpha}$ is a fixed indifferent direction.
	    \end{proof}

    \subsection{Counting non-attracting subtrees} \label{contodeisottalberi}
	    Let $R \in K(z)$ be a nonconstant rational function and $\alpha \in \hypb$.\\
	    In \cite{okuyama2025intrinsicreductionsintrinsicdepths} Okuyama defines the quantity
	    \begin{flalign*}
	        &&\dep_R v := R^*\delta_\zeta\left ( \bdisk(v) \right ),      &&\forall \ v \in T_\zeta\pwanb, \ \zeta \in \pwanb, 
	    \end{flalign*}
	    the \textbf{depth} of a direction $v$. In words, this is the number of preimages of the point $\zeta$ (counted with multiplicity) contained in the disk $\bdisk(v)$, for a direction $v$ at $\zeta$. Since the total number of preimages of a point $\zeta \in \pwanb$ is always $\deg R$ and $\pwanb = \{ \zeta \} \sqcup \left (  \bigsqcup_{v \in T_\zeta\pwanb} \bdisk(v) \right )$, we have the identity
	    \begin{flalign} \label{okuyamasformula}
	        &&\deg R = (\deg_R \zeta ) \times \uuno\left ( \zeta \text{ is fixed}\right ) + \sum_{v \in T_\zeta\pwanb} \dep_R v,  &&&\forall \ \zeta \in \pwanb .
	    \end{flalign}
	    \begin{defn} \label[defn]{tomorrowpot}
	        Define, for $\zeta \in \pwanb \smallsetminus \{\alpha\}$, the open set
	        \begin{equation}\label{defvalpha}
	            V^\alpha(\zeta) := \pwanb \smallsetminus \overline{\bdisk(\dir{\zeta \alpha})} = \bigsqcup_{u \in T_\zeta\pwanb \smallsetminus \{\dir{\zeta\alpha}\} } \bdisk(u) \minis .
	        \end{equation}
	        Define the function $t_R^\alpha: \pwanb \to \NN$ by
	        \begin{equation} \label{lilt}
	            t_R^\alpha(\zeta) := \begin{cases}
	                             R^* \delta_\zeta \left ( V^\alpha(\zeta) \right ) \ & \zeta \neq \alpha \\
	                             0 \                                          & \zeta = \alpha
	                        \end{cases} \, .
	        \end{equation}
	    \end{defn}

        \begin{prop} \label[prop]{contaccio}
           Let $\zeta_0 \in \pwanb, v \in T_{\zeta_0}\pwanb$. Then there exists $\zeta_v \in \bdisk(v)$ such that the restriction of $t_R^\alpha$ to the segment $(\zeta_0,\zeta_v)$ is constant and its value is given by
            \begin{equation}\label{comeffattatpiccolo}
                t_R^\alpha|_{(\zeta_0, \zeta_v)} = \begin{cases}
                    \dep_R v                                                                                                                                                                  & v \neq \dir{\zeta_0 \alpha} \\
                    \deg R - \dep_R \dir{\zeta_0 \alpha} - (\deg_R{\dir{\zeta_0\alpha}}) \times \uunononrepel{\dir{\zeta_0\alpha}} \quad      & v = \dir{\zeta_0 \alpha}
                \end{cases}\ .
            \end{equation}
            (Notice that when $\zeta_0=\alpha$, then $\dir{\zeta_0\alpha}=\dir{\alpha\alpha}=\alpha$ is not a direction and is therefore different from all $v\in T_\alpha\pwanb$; that is, the second line in the brace never applies when $\zeta_0=\alpha$.)
        \end{prop}
        \begin{proof}        
            Let $[\zeta_0,\zeta_L]$ be a segment stemming into $\bdisk(v)$. Let then $\varepsilon >0$, $\xi_i, \ \xi_i^k(r),\ \zeta_r$, for $i\in [m], \ k \in [\ell_i]$, and $r \in [0,\varepsilon]$, be as in \cref{preimageofsegment}.
            \begin{clamio}\label{scollateve}
                For all $\varepsilon>0$ small enough, we have
                \begin{equation}\label{implicagnoscollateve}
                    \dir{\xi_i \xi_i^k(\varepsilon)} \neq v \quad \implies \quad [\xi_i,\xi_i^k(\varepsilon)] \subset \pwanb \smallsetminus \bannulus \left (\zeta_0, \zeta_\varepsilon \right )  .
                \end{equation}
            \end{clamio}
            \begin{dimclamio}
                Suppose first that $\xi_i=\zeta_0$. Then $(\xi_i,\xi_i^k(\varepsilon)] \subset \bdisk(\dir{\zeta_0\xi_i^k(\varepsilon)})$, which is an open disk with boundary point $\xi_i = \zeta_0$. Therefore if  $\dir{\xi_i\xi_i^k(\varepsilon)} \neq v$, this disk is disjoint from $\overline{\bdisk(v = \dir{\zeta_0\zeta_\varepsilon})} \supset \overline{\bdisk(\dir{\zeta_0\zeta_\varepsilon})} \cap \overline{\bdisk(\dir{\zeta_\varepsilon\zeta_0})} = \overline{\bannulus(\zeta_0,\zeta_\varepsilon)}$, and we obtain the right-hand side in \cref{implicagnoscollateve}.\\
                We notice that if $0<s<r\leq \varepsilon$, then $\zeta_s \in (\zeta_0,\zeta_r)$ and by \cref{lemmabruttu} we have $\overline{\bdisk(\dir{\zeta_{s}\zeta_0})} \subset \overline{\bdisk(\dir{\zeta_r \zeta_0})}$, whence, by writing $\overline{\bannulus(\zeta_0,\zeta_r)}= \overline{\bdisk(\dir{\zeta_0\zeta_r})} \cap \overline{\bdisk(\dir{\zeta_r\zeta_0})} = \overline{\bdisk(\dir{\zeta_0\zeta_s})} \cap \overline{\bdisk(\dir{\zeta_r\zeta_0})} \supset \overline{\bdisk(\dir{\zeta_0\zeta_s})} \cap \overline{\bdisk(\dir{\zeta_s\zeta_0})}  =\overline{\bannulus(\zeta_0,\zeta_s)}$, we find that
                \begin{flalign}\label{monotoniadellecorone}
                    &&\pwanb \smallsetminus \overline{\bannulus(\zeta_0,\zeta_r)} \subset \pwanb \smallsetminus \overline{\bannulus (\zeta_0,\zeta_{s})}, && \forall \ s,r\in (0,\varepsilon], \ s < r.
                \end{flalign}
                Suppose now that $\xi_i \neq \zeta_0$. We will prove that for all $\varepsilon>0$ small enough we have
                \begin{equation}\label{preimplicagnoscollateve}
                	\dir{\xi_i \xi_i^k(\varepsilon)} \neq v \quad \implies \quad (\xi_i,\xi_i^k(\varepsilon)] \subset \pwanb \smallsetminus \overline{\bannulus \left (\zeta_0, \zeta_\varepsilon \right )} \minis .
                \end{equation}
                We then obtain \cref{implicagnoscollateve} by taking closures in the right-hand side.\\
                For $\varepsilon >0$ small enough we have $\xi_i \in \pwanb \smallsetminus \overline{\bannulus(\zeta_0,\zeta_\varepsilon)}$ by \cref{lemmadisgrazia}. The latter is an open set, so if $C$ is its connected component containing $\xi_i$, we will have an intersection $C \cap [\xi_i,\xi_i^k(\varepsilon)] \supset [\xi_i,\xi_i^k(\varepsilon')]$ for some $\varepsilon' \in (0,\varepsilon)$. By the monotonicity \cref{monotoniadellecorone}, it follows that $[\xi_i,\xi_i^k(\varepsilon')] \subset C  \subset \pwanb \smallsetminus \overline{\bannulus (\zeta_0,\zeta_{\varepsilon'})}$ for all $\varepsilon'>0$ small enough.                Since the $[m],[\ell_i]$ are finite sets, we can therefore repeat the argument until we find $\varepsilon >0$ small enough so that for all $i,k$ satisfying $\dir{\xi_i\xi_i^k(\varepsilon)} \neq v$ it holds that $(\xi_i,\xi_i^k(\varepsilon)]\subset \pwanb \smallsetminus \overline{\bannulus (\zeta_0,\zeta_{\varepsilon})}$. 
            \end{dimclamio}
            \noindent\textit{Back to the main proof (\cref{contaccio}).}
            \begin{caso}\label{zetalphaneqv}
                In this case we suppose that $v\neq \dir{\zeta_0\alpha}$.\\
                Since $v= \dir{\zeta_0\zeta_\varepsilon}$, by our \cref{zetalphaneqv} hypothesis, the segments $[\alpha,\zeta_0)$ and $(\zeta_0,\zeta_\varepsilon]$ are disjoint, whence by \cref{segmentsharingonept} we have that $\zeta_r\in (\zeta_0,\zeta_\varepsilon] \subset (\alpha,\zeta_\varepsilon]$ for all $r\in (0,\varepsilon]$. It follows that
                \begin{flalign}\label{direzionizetalphaneqv}
                    &&\dir{\zeta_r\zeta_0} =\dir{\zeta_r\alpha}, \quad \dir{\zeta_0\zeta_r} =\dir{\zeta_0\zeta_\varepsilon} = v\minis , && \forall \ r \in (0,\varepsilon] \minis .
                \end{flalign}
                \begin{clamio}\label{seguardiamov}
                    For all $\varepsilon>0$ small enough, if $\dir{\xi_i\xi_i^k(\varepsilon)} = v$ for some $i\in [m], \ k \in [\ell_i]$, then $[\xi_i,\xi_i^k(r)] \subset \overline{\bdisk \left ( \dir{\zeta_r\alpha}\right )}$ for all $r \in (0,\varepsilon)$. 
                \end{clamio}
                \begin{dimclamio}
                    Suppose that $\dir{\xi_i\xi_i^k(\varepsilon)} = v = \dir{\zeta_0\zeta_\varepsilon}$ for some $i\in [m], \ k \in [\ell_i]$; in particular, we have that $\zeta_0=\xi_i$ is a fixed point. If now $\zeta_0=\alpha$, then $\xi_i =\zeta_0=\alpha \in \hypb$ and is therefore not type I. Instead, if $\zeta_0\neq \alpha$, then $\dir{\zeta_0\alpha}$ is a direction and since we are in \cref{zetalphaneqv} it is distinct from $v$; therefore $T_{\zeta_0}\pwanb$ has at least two elements, so by \cref{classifitangent} we find again that $\zeta_0$ cannot be type I. By \cref{classifidinamica} we have thus shown, that whether $\zeta_0=\alpha$ or not, the point $\zeta_0$ is a non-attracting fixed point. It then follows from \cref{taledirezionetalepunto} (a) and \cref{dicotomiadirezioni} (b) that $v=\dir{\xi_i\xi_i^k(\varepsilon)}$ is a non-attracting fixed direction. Therefore, for $\varepsilon >0$ small enough, we have for all $r\in (0,\varepsilon)$, that $[\xi_i=\zeta_0, \xi_i^k(r)] \subset [\zeta_0,\zeta_r]$. The latter segment is contained in $\{\zeta_r\} \sqcup \bdisk(\dir{\zeta_r\zeta_0}) =\overline{\bdisk(\dir{\zeta_r\zeta_0})}$, which equals $\overline{\bdisk(\dir{\zeta_r\alpha})}$ by \cref{direzionizetalphaneqv}.
                \end{dimclamio}
                \begin{clamio}\label{bastaguardareglixiconi}
                    For all $\varepsilon>0$ small enough, for all $r \in (0,\varepsilon), \ i \in [m], \ k\in [\ell_i]$, we have
                    \begin{equation}\label{implicagnobastaguardareglixiconi}
                        \xi_i \in \bdisk(v) \quad \iff \quad \xi_i^k(r) \in V^\alpha(\zeta_r)\minis .
                    \end{equation}
                \end{clamio}
                \begin{dimclamio}
                    We will show that for each pair $i\in [m], \ k\in [\ell_i]$, if $\varepsilon>0$ is small enough, the right- and left-hand side of \cref{implicagnobastaguardareglixiconi} are either both true or both false for all $r\in (0,\varepsilon)$. Since $[m]$ and $[\ell_i]$ are finite sets, this is enough to show that for small enough $\varepsilon>0$ the double implication \cref{implicagnobastaguardareglixiconi} holds for all pairs $i,k$ at once.\\
                    Suppose first that $\dir{\xi_i\xi_i^k(\varepsilon)} = v$. Then $\xi_i = \zeta_0$, so the left-hand side in \cref{implicagnobastaguardareglixiconi} is false, and since $\overline{\bdisk(\dir{\zeta_r\alpha})}$ is the complement of $V^\alpha(\zeta_r)$, by \cref{seguardiamov} the right-hand side is false as well.\\ 
                    Suppose now that $\dir{\xi_i\xi_i^k(\varepsilon)} \neq v$. Since $\zeta_r\in (\zeta_0,\zeta_\varepsilon)$, by \cref{lemmabruttu} we have $\overline{\bdisk(\dir{\zeta_r \zeta_0})} \subset \bdisk(\dir{\zeta_\varepsilon\zeta_0})$, from which it follows that 
                    \begin{equation}\label{preinclugio}
                        \bannulus(\zeta_0,\zeta_r) \sqcup \{\zeta_r\} = \bdisk(\dir{\zeta_0\zeta_r}) \cap \overline{\bdisk(\dir{\zeta_r\zeta_0})} \subset \bdisk(\dir{\zeta_0\zeta_r}) \cap \bdisk(\dir{\zeta_\varepsilon\zeta_0}) = \bannulus(\zeta_0,\zeta_\varepsilon) \minis .
                    \end{equation}
                    Using \cref{scollateve} and \cref{preinclugio}, we have that for all $\varepsilon >0$ small enough, 
                    \begin{multline}\label{inclugio}
                        [\xi_i, \xi_i^k(\varepsilon)] \subset \pwanb \smallsetminus \bannulus(\zeta_0,\zeta_\varepsilon) \subset  \pwanb \smallsetminus \left (\bannulus(\zeta_0,\zeta_r) \sqcup \{\zeta_r\}\right )
                        = \left ( \pwanb \smallsetminus \bdisk\left (\dir{\zeta_0\zeta_r}\right ) \right ) \cup \left ( \pwanb \smallsetminus \overline{\bdisk\left (\dir{\zeta_r\zeta_0}\right )} \right ) \\
                        = \left ( \pwanb \smallsetminus \bdisk\left (v\right ) \right ) \cup \left ( \pwanb \smallsetminus \overline{\bdisk\left (\dir{\zeta_r\alpha}\right )} \right ), \qquad \forall \ r \in (0,\varepsilon) \minis .
                    \end{multline}
                    In the last equality we used \cref{direzionizetalphaneqv}. 
                    Now, $\bdisk(v),\bdisk(\dir{\zeta_r\alpha})$ are open disks, so by definition their complement is nonempty. By \cref{lemmaobbrobrioso}, the closed sets $\pwanb \smallsetminus \bdisk(v) = \pwanb \smallsetminus \bdisk(\dir{\zeta_0\zeta_r})$ and $\pwanb \smallsetminus \bdisk(\dir{\zeta_r\alpha}) = \pwanb \smallsetminus \bdisk(\dir{\zeta_r\zeta_0})$ are disjoint. It then follows from the inclusion \cref{inclugio} that the segment $[\xi_i,\xi_i^k(\varepsilon)]$, being connected, must be completely contained in either $\pwanb \smallsetminus \bdisk(v)$ or $\pwanb \smallsetminus \bdisk(\dir{\zeta_r\alpha})$, and in fact, in either $\pwanb \smallsetminus \bdisk(v)$ or in $\pwanb \smallsetminus \overline{\bdisk(\dir{\zeta_r\alpha})} = V^\alpha(\zeta_r)$. This shows that the conditions in \cref{implicagnobastaguardareglixiconi} either both hold at once or neither holds.
                \end{dimclamio}
                \noindent \textit{Back to the main proof (\cref{contaccio}).}
                Let $\varepsilon>0$ be small enough for the conclusion of \cref{bastaguardareglixiconi} to hold. Then, since $V^\alpha(\zeta_r)$ is the complement of $\overline{\bdisk(\dir{\zeta_r\alpha})}$, the double implication \cref{implicagnobastaguardareglixiconi} shows that for $r\in (0,\varepsilon)$ the first and second line in the following display are equal:
                \begin{multline*} 
                    t^\alpha_R(\zeta_r) = R^*\delta_{\zeta_r}(V^\alpha(\zeta_r)) = \sum_{\xi_i^k(r) \in V^\alpha(\zeta_r)} \deg_R \xi_i^k(r)  \\
                    = \sum_{\substack{i \in [m] \\ : \ \xi_i \in \bdisk(v)}} \sum_{k \in [\ell_i]} \deg_R \xi_i^k(r) \ = \sum_{\substack{i \in [m] \\ : \ \xi_i \in \bdisk(v)}} \deg_R \xi_i = \dep_R v \minis .
                \end{multline*} 
                (In the second line we use \cref{degreessumup} (a) together with the fact that $\deg_R \dir{\xi_i\xi_i^k(\varepsilon)} = \deg_R \xi_i^k(r)$, for all $i\in [m], \ k \in [\ell_i]$, which we had from \cref{preimageofsegment}.) This shows that if we take $\zeta_v= \zeta_\varepsilon$, then \cref{comeffattatpiccolo} holds if $v\neq \dir{\zeta_0\alpha}$. This concludes the proof for \cref{zetalphaneqv}.
            \end{caso}
            \begin{caso}\label{vzetazeroalpha}
                In this case we suppose that $v = \dir{\zeta_0\alpha}$. Then we are assuming in particular that $\dir{\zeta_0\alpha}$ is a direction, and therefore that $\zeta_0 \neq \alpha$. Then for all $\varepsilon >0$ small enough we have
                \begin{equation}\label{situasegmentivzetazeroalpha}
                    [\zeta_0,\zeta_\varepsilon] \subset [\zeta_0,\alpha) \minis .
                \end{equation}
                For $r \in (0, \varepsilon)$, we have $\pwanb \smallsetminus V^\alpha(\zeta_r) = \overline{\bdisk(\dir{\zeta_r\alpha})}$, so that, to obtain the equality with the second line in \cref{comeffattatpiccolo}, we can equivalently show that for small enough $\varepsilon >0$,
	            \begin{equation}\label{possiamodimostrarequesto}
	                R^*\delta_{\zeta_r}\left ( \overline{\bdisk \left (\dir{\zeta_r\alpha} \right )}\right ) = \, \dep_R \dir{\zeta_0\alpha} \, + \minis \left (\deg_R \dir{\zeta_0\alpha}\right )\times \uunononrepel{\dir{\zeta_0\alpha}},  \qquad \forall \ r \in (0,\varepsilon) ,
	            \end{equation}
	            and choose $\zeta_\varepsilon$ as our $\zeta_v$.
	            \begin{clamio}\label{nonzetazeroalpha}
	                For all $\varepsilon >0$ small enough, for all $r\in (0,\varepsilon)$, we have
	                \begin{equation}\label{nonzetazeroalphauguaglio}
	                    \Set{\xi_i^k(r) \ | \ \xi_i \in \bdisk \left ( \dir{\zeta_0\alpha} \right)} \ = \ \Set{\xi_i^k(r) \ | \ \dir{\xi_i\xi_i^k(\varepsilon)} \neq \dir{\zeta_0\alpha}} \cap \overline{\bdisk\left (\dir{\zeta_r\alpha} \right )}\minis .
	                \end{equation}
	            \end{clamio}
	            \begin{dimclamio}
	                Let $r \in (0,\varepsilon)$. Since $\zeta_r \in (\zeta_0,\zeta_\varepsilon)$, it follows from \cref{lemmaobbrobrioso} that $\pwanb \smallsetminus \bdisk \left (\dir{\zeta_\varepsilon\zeta_r} \right ) \,  = \, \bdisk \left (\dir{\zeta_r\zeta_\varepsilon} \right ) \smallsetminus \bdisk \left (\dir{\zeta_r\zeta_\varepsilon} \right ) \ \subset \ \overline{\bdisk \left (\dir{\zeta_r\zeta_\varepsilon} \right )}$. Since $\zeta_r \in (\zeta_0,\zeta_\varepsilon)$ we have that $\dir{\zeta_\varepsilon \zeta_r} =  \dir{\zeta_\varepsilon \zeta_0}$ and taking $\varepsilon >0$ small enough for \cref{situasegmentivzetazeroalpha} to hold, we also have that $\dir{\zeta_r \zeta_\varepsilon} = \dir{\zeta_r\alpha}$. Therefore, for all small enough $\varepsilon >0$, 
	                \begin{flalign}\label{complementonelzetaralpha}
	                    &&\pwanb \smallsetminus \bdisk \left (\dir{\zeta_\varepsilon\zeta_0} \right ) \, \subset \, \overline{\bdisk \left (\dir{\zeta_r\alpha} \right )}, && \forall \ r \in (0,\varepsilon) \minis .
	                \end{flalign}
	                Recalling our hypothesis that $\dir{\zeta_0\zeta_\varepsilon}= \dir{\zeta_0\alpha}$, we find that
	                \begin{multline}
	                     \pwanb \smallsetminus \bannulus(\zeta_0,\zeta_\varepsilon) = \left ( \pwanb \smallsetminus \bdisk \left (\dir{\zeta_0\zeta_\varepsilon} \right )\right ) \cup \left ( \pwanb \smallsetminus \bdisk \left (\dir{\zeta_\varepsilon \zeta_0}\right ) \right ) \\
	                    \subset \left ( \pwanb \smallsetminus \bdisk \left (\dir{\zeta_0\alpha}\right )\right ) \cup \overline{\bdisk\left (\dir{\zeta_r \alpha}\right )}  ,  \qquad \forall \ r \in (0,\varepsilon) \minis .
	                \end{multline}
	                Now, for all $\varepsilon > 0$ small enough we have \cref{implicagnoscollateve}, which together with the inclusion above and our hypothesis that $v = \dir{\zeta_0\alpha}$, gives us
	                \begin{flalign} \label{inclugiozetazeroalpha}
	                    && \dir{\xi_i\xi_i^k(\varepsilon)} \neq \dir{\zeta_0\alpha} \quad \implies \quad  [\xi_i,\xi_i^k(\varepsilon)]
	                    \subset \left ( \pwanb \smallsetminus \bdisk \left (\dir{\zeta_0\alpha}\right )\right ) \cup \overline{\bdisk\left (\dir{\zeta_r \alpha}\right )} , && \forall \ r \in (0,\varepsilon) \minis .
	                \end{flalign}
	                The set $\pwanb \smallsetminus \bdisk (\dir{\zeta_0\alpha})$ is nonempty because it is the complement of an open disk, and the set $\overline{\bdisk\left (\dir{\zeta_r \alpha}\right )}$ is nonempty because it is the closure of a open disk. By \cref{lemmabruttu} and \cref{situasegmentivzetazeroalpha} we have that $\overline{\bdisk(\dir{\zeta_r \alpha})} \subset \bdisk(\dir{\zeta_0\alpha})$, whence $\left ( \pwanb \smallsetminus \bdisk \left (\dir{\zeta_0\alpha}\right )\right ) \cup \overline{\bdisk\left (\dir{\zeta_r \alpha}\right )}$ is a disjoint union of nonempty  closed sets.   Therefore in the right-hand side of \cref{inclugiozetazeroalpha}, since $[\xi_i,\xi_i^k(\varepsilon)]$ is connected, it is completely contained in either $\pwanb \smallsetminus \bdisk (\dir{\zeta_0\alpha})$ or $\overline{\bdisk\left (\dir{\zeta_r \alpha}\right )}$. Now, if $\xi_i^k(r)$ is an element of either member of the equality \cref{nonzetazeroalphauguaglio}, then we have $\dir{\xi_i\xi_i^k(\varepsilon)} \neq \dir{\zeta_0\alpha}$, and hence by \cref{inclugiozetazeroalpha} and our discussion, the segment $[\xi_i,\xi_i^k(\varepsilon)]$ must either be completely contained in $\overline{\bdisk(\dir{\zeta_r\alpha})}$ or it must be disjoint from $\bdisk(\dir{\zeta_0\alpha})$. This proves the equality \cref{nonzetazeroalphauguaglio}. 
	            \end{dimclamio}
	            \noindent\textit{Back to the main proof (\cref{contaccio}).} It follows from \cref{nonzetazeroalpha} that for all $\varepsilon >0$ small enough, we have the first equality in the following display: 
	            \begin{multline} \label{contoapartezetazeroalfa}
	                R^* \delta_{\zeta_r} \left (  \Set{ \xi_i^k(r) \ | \ \dir{\xi_i\xi_i^k(\varepsilon)} \neq \dir{\zeta_0\alpha}} \cap \overline{\bdisk \left (  \dir{\zeta_r\alpha}\right )} \right ) 
	                = \sum_{\substack{i\in [m] \ : \\ \xi_i \in \bdisk(\dir{\zeta_0 \alpha})}} \sum_{k \in [\ell_i]} \deg_R \xi_i^k(r) \\
	                = \sum_{\substack{i \in [m] \ : \\ \xi_i \in \bdisk(\dir{\zeta_0\alpha})}} \deg_R \xi_i = \dep_R \dir{\zeta_0\alpha} \minis , \qquad \forall \ r \in (0,\varepsilon) \minis .
	            \end{multline}
	            (In going to the second line we used \cref{degreessumup} (a), and the fact that $\deg_R \xi_i^k(r) = \deg_R \dir{\xi_i \xi_i^k(\varepsilon)}$ for all $r \in (0,\varepsilon)$, which we had from \cref{preimageofsegment}.) 
	
	            Suppose now that for some $i\in [m], \ k\in [\ell_i]$, we have $\dir{\xi_i\xi_i^k(\varepsilon)} = \dir{\zeta_0\alpha}$. Then by our \cref{vzetazeroalpha} hypothesis, $\dir{\xi_i\xi_i^k(\varepsilon)} = \dir{\zeta_0\alpha}$ is a fixed direction and in particular $\zeta_0 = \xi_i$ is a fixed point. For all $\varepsilon > 0$ small enough we have $[\xi_i, \xi_i^k(\varepsilon)]\subset [\zeta_0,\alpha]$. By \cref{dicotomiadirezioni} (a), up to picking a smaller $\varepsilon>0$, we have that
	            \begin{equation}\label{repulzoono}
	               	\begin{array}{rcl}
	               		\text{either } \dir{\zeta_0\alpha} \text{ is non-repelling, i.e.,} &	[\zeta_0,\zeta_r] \subset [\zeta_0,\xi_i^k(r)],	&\forall \ r \in (0,\varepsilon), \\
	               		\text{or } \dir{\zeta_0\alpha} \text{ is repelling, i.e.,} & [\zeta_0,\zeta_r) \supset [\zeta_0,\xi_i^k(r)],	& \forall \ r\in (0,\varepsilon).
	               	\end{array}
	            \end{equation} 
	            Now, for all small enough $\varepsilon>0$ we have $[\xi_i = \zeta_0,\xi_i^k(\varepsilon)] \subset [\zeta_0,\alpha]$. For such $\varepsilon>0$, if $\dir{\zeta_0\alpha}$ is non-repelling, then by \cref{lemmabruttu} and \cref{repulzoono} we have $\xi_i^k(r) \in \overline{\bdisk(\dir{\xi_i^k(r)})} \subset  \overline{\bdisk(\dir{\zeta_r\alpha})}$ for all $r\in (0,\varepsilon)$. In \cref{situasegmentivzetazeroalpha} we had taken small enough $\varepsilon>0$ to have $(\zeta_0 = \xi_i, \zeta_\varepsilon] \subset (\zeta_0, \alpha]$. For such $\varepsilon >0$, if $\dir{\zeta_0\alpha}$ is repelling, then by \cref{lemmabruttu} and \cref{repulzoono} we have $\overline{\bdisk(\dir{\zeta_r \alpha})} \subset \bdisk(\dir{\xi_i^k(r) \alpha}) \not \ni \xi_i^k(r)$.
	            It follows from this discussion that
	            \begin{equation*}
	               	\dir{\xi_i\xi_i^k(\varepsilon)} = \dir{\zeta_0\alpha} \quad \implies \quad \left ( \xi_i^k(r) \in  \overline{\bdisk(\dir{\zeta_r\alpha})}  \ \iff \ \substack{\dir{\zeta_0\alpha} \text{ is a non-repelling} \\ \text{fixed direction}} \right )
	            \end{equation*}
	             Since we also had that $\deg_R \xi_i^k(r) = \deg_R \dir{\xi_i\xi_i^k(\varepsilon)} = \deg_R \dir{\zeta_0\alpha}$ for all $r\in (0,\varepsilon)$, from \cref{preimageofsegment}, we conclude that for all $\varepsilon >0$ small enough,
	            \begin{multline*}
	                R^* \delta_{\zeta_r} \left (  \Set{ \xi_i^k(r) \ | \ \dir{\xi_i\xi_i^k(\varepsilon)} = \dir{\zeta_0\alpha}} \cap \overline{\bdisk \left (  \dir{\zeta_r\alpha}\right )} \right ) \\
	                =  (\deg_R \dir{\zeta_0\alpha}) \times \uunononrepel{\dir{\zeta_0\alpha}} , \qquad \forall \ r\in(0,\varepsilon)  .
	            \end{multline*}
	            Summing this equality to \cref{contoapartezetazeroalfa}, we obtain \cref{possiamodimostrarequesto}. \qedhere
	        \end{caso}
	    \end{proof}
        \begin{defn}
            We define a the function $T_R^\alpha : \hypb \to \RR_{\geq 0}$ by
            \begin{equation} \label{bigt}
                T_R^\alpha (\zeta) :=  \int_{\alpha}^\zeta t_R^\alpha d \lambda \minis .
            \end{equation}
        \end{defn}
        We will show in the proof of the theorem below that $t^\alpha_R$ is $\lambda$-a.e. decreasing on segments $(\alpha,\zeta)$, so in particular the integral above is well-defined.
        %\begin{propdef} \label[propdef]{contolaplaciazzo}
        \begin{theorem}[manual-num=\ref{tausuigamma}]
            The function $T_R^\alpha$ belongs to $\operatorname{BDV}(\pwanb)$ and there exists a positive finite Radon measure $\tau_R^\alpha$ of total mass $\deg R$ satisfying 
             \begin{equation}\label{laplaciazzoT}
                \tau_R^\alpha = \Delta T_R^\alpha + (\deg R) \times \delta_\alpha =       \sum_{\Gamma \in \nas(R)} \left ( (r_\Gamma)_*\delta_{\alpha} \ + \  \sum_{\zeta \in \Gamma} \left ( \deg_R \zeta - 1 \right ) \delta_\zeta  \right )   .
            \end{equation}
            In particular, the measure $\tau^\alpha_R$ charges each $\Gamma \in \nas(R)$ with mass $\deg_R\Gamma$ and
            \begin{equation}\label{sommadeigradideipinaffoli}
                \sum_{\Gamma \in \nas(R)} \deg_R \Gamma \; = \; \deg R \minis .
            \end{equation}
        \end{theorem}
        \begin{proof}
            Throughout the proof, we will write $T= T_R^\alpha, t = t_R^\alpha, \tau =\tau_R^\alpha$, thus omitting the dependence from $R$ and $\alpha$. For $\zeta \in \pwanb, v \in T_\zeta\pwanb$, we denote by $t(v)$ the value of $t$ on the segment $(\zeta,\zeta_v)$ in the direction $v$ from \cref{contaccio}.
            
            Recalling that $\dir{\alpha\alpha}=\alpha$, it follows from \cref{contaccio} and \cref{okuyamasformula} that
            \begin{flalign}\label{bucogauss}
                t(\dir{\alpha\alpha}) \ - \sum_{v \in T_\alpha \pwanb \smallsetminus \{\dir{\zeta\alpha}\}} t(v) \ = -\sum_{v \in T_\alpha \pwanb } t(v) \ = -\sum_{v \in T_{\alpha}\pwanb} \dep v \ =  (\deg_R \alpha ) \times \uuno   \left ( \alpha \text{ is fixed} \right ) -\deg R \, .
            \end{flalign}
            Suppose now $\zeta\in \fix_{<1}(R)$. Then $\zeta$ is type I by \cref{classifidinamica} (a), so $\zeta \neq \alpha \in \hypb$, whence $\dir{\zeta\alpha}$ is a direction. By \cref{classifitangent}  the point $\zeta$ only has one direction, $\dir{\zeta \alpha}$, which is therefore fixed by the map on tangent spaces $R: T_\zeta\pwanb \to T_\zeta\pwanb = \{\dir{\zeta\alpha}\}$, whence by \cref{degreessumup} (a) we have $\deg_R \zeta = \deg_R \dir{\zeta\alpha}$. Since $\zeta$ is attracting, by \cref{taledirezionetalepunto} (a) the fixed direction $\dir{\zeta\alpha}$ is attracting, and is therefore non-repelling by \cref{dicotomiadirezioni}. From this discussion,  \cref{contaccio} and \cref{okuyamasformula}, it follows that
            \begin{multline}\label{piattofissoatt}
                t(\dir{\zeta\alpha}) \ - \sum_{v \in T_\zeta\pwanb \smallsetminus \{\dir{\zeta\alpha}\}} t(v) = \  t(\dir{\zeta\alpha}) \
                =  \ \deg R \ - \deg_R \dir{\zeta\alpha} \ - \dep_R \dir{\zeta\alpha}    \ \\
                = \ \deg R \ - \deg_R \zeta \ - \dep_R \dir{\zeta\alpha} \ = \ 0, \qquad  \forall \ \zeta \in \fix_{<1}(R)   .
            \end{multline}
            If $\zeta \in \pwanb \smallsetminus ( \fix(R) \cup \{\alpha \} )$, then no direction at $\zeta$ is fixed, so that, again by Propostion \ref{contaccio} and \cref{okuyamasformula},
            \begin{multline}\label{piattononfisso}
                t(\dir{\zeta\alpha}) \ - \sum_{v \in T_\zeta \pwanb \smallsetminus \{\dir{\zeta\alpha}\}} t(v) \ = \ \deg R \ -\dep_R \dir{\zeta \alpha} \ - \sum_{v \in T_\zeta \pwanb \smallsetminus \{ \dir{\zeta \alpha}\}} \dep_R v \ = \ 0 , \\
                \forall \ \zeta \in \pwanb \smallsetminus ( \fix(R) \cup \{\alpha \}).
            \end{multline}
            If now $\zeta \in \nafl(R)$ and is type I, then $\zeta \neq \alpha \in \hypb$, so by \cref{degreessumup} (a) we have $\deg_R \dir{\zeta\alpha} \leq \deg_R \zeta$. By \cref{gradiinnafl} (b) we then have $1 = \deg_R \zeta = \deg_R \dir{\zeta\alpha}$. Instead, if $\zeta \in \nafl(R) \smallsetminus \{\alpha\}$, is not type I, and $\dir{\zeta \alpha}$ is a non-repelling fixed direction, then $\deg_R \dir{\zeta \alpha} = 1$ by \cref{indifferentdirections} (c). It follows that  
            \begin{flalign*}
                &&(\deg_R \dir{\zeta \alpha}) \times \uuno \left (\substack{\dir{\zeta \alpha} \text{ is a non-repelling} \\ \text{fixed direction}} \right ) \ = \  \uuno \left (\substack{\dir{\zeta \alpha} \text{ is a non-repelling} \\ \text{fixed direction}} \right ), && \forall \ \zeta \in \nafl(R) \smallsetminus \{\alpha \},
            \end{flalign*}
            whence by \cref{contaccio} and \cref{okuyamasformula},
            \begin{multline}  \label{bozzononatt}
                t(\dir{\zeta\alpha}) \ - \sum_{v \in T_\zeta\pwanb \smallsetminus \{\dir{\zeta\alpha}\} } t(v) 
                = \  \deg R \ - \dep_R \dir{\zeta \alpha} \ - \uuno \left ( \substack{\dir{\zeta\alpha} \text{ is a non-repelling} \\ \text{fixed direction}} \right ) \ - \sum_{v \in T_\zeta \pwanb \smallsetminus \{\dir{\zeta \alpha}\}} \dep_R v \\
                 = \ \deg_R \zeta - \ \uuno \left ( \substack{ \dir{\zeta\alpha} \text{ is a non-repelling}  \\ \text{fixed direction}}\right )  ,  \qquad \qquad \forall \ \zeta \in \nafl(R) \smallsetminus \{\alpha \} \minis .
            \end{multline}
            The right-hand sides in (\ref{piattofissoatt}, \ref{piattononfisso}, \ref{bozzononatt}) are all nonnegative, hence
            \begin{flalign}\label{decrescentina}
                &&0 \leq  \sum_{v \in T_\zeta \pwanb \smallsetminus \{\dir{\zeta \alpha}\}} t(v)  \leq  t  (  \dir{\zeta\alpha}  ) ,  &&\forall \ \zeta \in \pwanb \smallsetminus \{\alpha \} \minis .
            \end{flalign}
            Let $\zeta \in \pwanb \smallsetminus\{\alpha\}$, let $v \in T_\zeta\pwanb \smallsetminus \{\dir{\zeta \alpha}\}$, and let $\zeta_v, \zeta_{ \dir{\zeta\alpha}}$ be as in \cref{contaccio}. Then by the above inequality, we have that $t$ is Lebesgue$\text{-a.e.}$ decreasing on the segment $(\zeta_{\dir{\zeta\alpha}}, \zeta_v )$. By using an open covering, it then follows that on any segment $(\alpha, \zeta], \ \zeta\in \hypb$, the function $t$ is Lebesgue-a.e. decreasing, with respect to the order on $(\alpha,\zeta]$ in which $\zeta$ is the maximal element. Without loss of generality, since our statement is about $T$, we can in fact assume that $t$ is decreasing, on all segments $(\alpha,\zeta)$. It follows that  
            \begin{equation}\label{direzioneschiaccina}
                0 \leq t(\xi) \leq t(v), \qquad \forall \ \zeta \in \pwanb , \ v \in T_\zeta\pwanb \smallsetminus \{\dir{\zeta\alpha}\},\ \xi \in \bdisk(v)  . 
            \end{equation}
            Now, if $t$ is constant on a segment $(\zeta,\zeta'') \subset (\alpha,\zeta'')$, then for any $\zeta' \in (\zeta,\zeta'')$, we have $t( \dir{\zeta'\alpha}) = t({\dir{\zeta' \zeta}}) = t({\dir{\zeta' \zeta''}})$, so that by the inequality \cref{decrescentina}, it follows that $t(v)=0$ for all $v \in T_{\zeta'}\pwanb \smallsetminus \{\dir{\zeta'\alpha}, \dir{\zeta' \zeta''} \}$. Considering \cref{direzioneschiaccina} and that $\bannulus(\zeta,\zeta'') \smallsetminus (\zeta,\zeta'') = \bigsqcup_{\zeta'\in(\zeta,\zeta'')} \bigsqcup_{v\in \pwanb \smallsetminus \{\dir{\zeta'\zeta},\dir{\zeta'\zeta''}\}} \bdisk(v)$,  it then follows that
            \begin{equation}\label{zerofuoridalloscheletro}
                \left ((\zeta,\zeta'')\subset (\alpha,\zeta''), \ t|_{(\zeta,\zeta'')} \text{ constant} \right ) \ \implies \ t|_{\bannulus(\zeta,\zeta'') \smallsetminus (\zeta,\zeta')} = 0 \minis .
            \end{equation}
            Similarly, it follows from \cref{decrescentina}, from \cref{bucogauss}, and from the fact that $t$ is integer valued, that for $\zeta \in \pwanb$ all but finitely many directions $v\in T_\zeta\pwanb$ satisfy $t(v) =0$, so that by \cref{direzioneschiaccina} we have
            \begin{flalign}\label{cofinitozero}
                \big \{v\in T_\zeta\pwanb \ \bigm | \  t(v) = 0 \big \} = \big \{v\in T_\zeta\pwanb  \ \bigm  | \ t|_{\bdisk(v)} \equiv 0 \big \} \text{ has a finite complement in } T_\zeta\pwanb,  \qquad  \forall \ \zeta \in \pwanb .&& 
            \end{flalign}
            Now, fixing a $\zeta \in \pwanb$, we can find $\zeta_v$ as in \cref{contaccio} for all finitely many directions $v \in T_\zeta \pwanb$ such that $t(v) >0$, and denote these $\zeta_v$ by $\xi_1,\ldots , \xi_k$. Then by \cref{zerofuoridalloscheletro} and \cref{cofinitozero} we can compose an open affinoid as in \cref{affinoidesmontato},
            \begin{equation*}
                U_\zeta = \{\zeta\} \sqcup \bannulus(\zeta, \xi_1) \sqcup \ldots \sqcup \bannulus (\zeta, \xi_k) \sqcup \left ( \bigsqcup_{v \in T_\zeta\pwanb \smallsetminus \{\dir{\zeta\xi_1}, \ldots , \dir{\zeta\xi_k}\}} \bdisk(v) \right )  = \aff{\xi_1, \ldots , \xi_k} 
            \end{equation*}
            which is a neighbourhood of $\zeta$ where $t$ is zero everywhere outside $\Sigma_\zeta = \spando\{\zeta, \xi_1, \ldots , \xi_k\}$. Since $\zeta \in \pwanb$ was arbitrary, we can in fact cover $\pwanb$ with open sets $U_\zeta$ as above, and in fact with finitely many of these, because $\pwanb$ is compact.  It follows that there exists a quasi-finite subgraph $\Sigma = \spando \{ \zeta_1, \ldots , \zeta_r\} \subset \pwanb$ outside of which $t$ vanishes. Up to enlarging $\Sigma$, we can assume $r \geq 2$ and $\alpha \in \Sigma^\circ$. By \cref{fogliespandono} we can take $\zeta_1, \ldots ,\zeta_r$ to be the points with only one direction in $\Sigma$, so that 
            \begin{flalign}\label{zetaconipuntanoalpha}
                &&T_{\zeta_i} \Sigma = \{\dir{\zeta_i \alpha} \}, &&\forall \ i \in \{ 1, \ldots ,r \}.
            \end{flalign}
            We have
            \begin{equation}\label{zerofuoridasigma}
                t(\zeta) = t(v) = 0,  \quad \forall \  \zeta \in \pwanb \smallsetminus \Sigma ,  \ v \in T_\zeta\pwanb, \quad \text{and}  \quad t(w) = 0, \quad  \forall \ w \in T_\zeta \pwanb \smallsetminus T_\zeta \Sigma, \ \zeta \in \Sigma \, . 
            \end{equation}
            Given our quasi-finite graph $\Sigma$ and our points $\zeta_1, \ldots ,\zeta_r$ with only one direction in $\Sigma$, let $\zeta_1^\circ, \ldots , \zeta_r^\circ$ be as in \cref{approxzetacirco}. By \cref{approxzetacirco} (a) we have in particular that
            \begin{flalign}\label{scomponisuzetacirco}
                &&(\zeta_i,\zeta_j) = (\zeta_i,\zeta_i^\circ] \sqcup (\zeta_i^\circ,\zeta_j^\circ) \sqcup [\zeta_j^\circ, \zeta_j), && \forall \ i,j\in \{1, \ldots, r \}, \ i \neq j .
            \end{flalign}
            Since $\alpha \in (\Sigma)^\circ = \bigcup_{1 \leq i <j \leq r} (\zeta_i,\zeta_j)$, there exist $i, j \in \{1, \ldots, r\} , \ i\neq j$ such that $\alpha \in (\zeta_i,\zeta_j)$. By \cref{scomponisuzetacirco}, up to replacing $\zeta_i^\circ$ with another point of $(\zeta_i, \zeta_i^\circ)$ and $\zeta_j^\circ$ with another point in $(\zeta_j^\circ, \zeta_j)$, we can assume that $\alpha \in (\zeta_i^\circ,\zeta_j^\circ)$. Doing this for all such $i \neq j$, since  $\alpha \in(\Sigma)^\circ = \bigcup_{j \in \{1, \ldots , r \}\smallsetminus \{i\}} (\zeta_i, \zeta_j)$ for all $i$, we have
            \begin{equation}\label{alphastainunijcirco}
                \forall \ i \in \{1, \ldots , r\}, \: \exists \ j \in \{1, \ldots, r\}\smallsetminus \{i\}, \quad \text{such that } \quad  \alpha \in (\zeta_i^\circ,\zeta_j^\circ).
            \end{equation}
            Let now again $i \in \{1, \ldots, r \}$, and let $\zeta\in (\zeta_i, \zeta_i^\circ)$. Then $\dir{\zeta\zeta_i} \neq \dir{\zeta\zeta_i^\circ}$, and by \cref{scomponisuzetacirco} we have $\dir{\zeta\zeta_j} =\dir{\zeta\zeta_i^\circ}$, for all $j\in \{1, \ldots , r\} \smallsetminus \{i\}$. Moreover, from \cref{scomponisuzetacirco} and \cref{alphastainunijcirco} together, it follows that $\dir{\zeta\zeta_i^\circ} =\dir{\zeta\alpha}$. By \cref{tangenteinterno} we conclude that $T_\zeta\Sigma$ consists of exactly two distinct directions, namely $\dir{\zeta\zeta_i}$ and $\dir{\zeta\alpha}$, for all $\zeta\in (\zeta_i,\zeta_i^\circ)$. Then, up to replacing $\zeta_i^\circ$ with another point in $(\zeta_i,\zeta_i^\circ)$, in fact we obtain
            \begin{flalign}\label{spazitangentisuiramoscelli}
                &&T_\zeta\Sigma = \{\dir{\zeta\zeta_i}, \dir{\zeta\alpha}\}, \quad \dir{\zeta\zeta_i} \neq \dir{\zeta\alpha}, && \forall \ \zeta \in (\zeta_i, \zeta_i^\circ], \ \forall \ i \in \{1, \ldots , r \}.
            \end{flalign}
            Let $i \in \{1, \ldots, r\}$. Again up to replacing $\zeta_i^\circ$ with another point in $(\zeta_i, \zeta_i^\circ)$, by \cref{contaccio}, we can assume that $t$ is constant on $(\zeta_i,\zeta_i^\circ)$. In particular, by \cref{zetaconipuntanoalpha} and \cref{spazitangentisuiramoscelli}, we have that $t(\dir{\zeta_i\alpha}) = t(\dir{\zeta\zeta_i}) = t(\zeta) = t(\dir{\zeta\alpha})$ for all $\zeta\in (\zeta_i, \zeta_i^\circ)$. It follows from \cref{spazitangentisuiramoscelli} that $\dir{\zeta\zeta_i^\circ} = \dir{\zeta\alpha}$ for all $\zeta\in (\zeta_i,\zeta_i^\circ)$, whence, up to replacing $\zeta_i^\circ$ with another point in $(\zeta_i, \zeta_i^\circ)$, we can in fact write
            \begin{flalign}\label{tcostantesuiramoscelli}
                &&t(\dir{\zeta_i\alpha}) = t(\dir{\zeta\zeta_i}) = t(\zeta) = t(\dir{\zeta\alpha}), && \forall \ \zeta \in (\zeta_i,\zeta_i^\circ], \ \forall \ i \in \{1, \ldots, r\}.
            \end{flalign}
            From \cref{spazitangentisuiramoscelli} and \cref{tcostantesuiramoscelli}, we obtain
            \begin{flalign}\label{piattosuiramoscelli}
                &&t(\dir{\zeta\alpha}) \ - \sum_{v \in T_\zeta \Sigma \smallsetminus \{\dir{\zeta\alpha}\}} t(v) \ = \ 0, && \forall \ \zeta \in (\zeta_i,\zeta_i^\circ], \ \forall \ i \in \{1,\ldots, r \}.  
            \end{flalign}
            Let now        
            \begin{equation*}
                \Sigma^{\ker} := \spando \{\zeta_1^\circ, \ldots , \zeta_r^\circ\} \minis .
            \end{equation*} 
            We then have from \cref{alphastainunijcirco} that 
            \begin{equation}\label{alphastainsigmacare}
                \alpha \in (\Sigma^{\ker})^\circ .
            \end{equation}
            Let $\widetilde\Sigma$ be a finite subgraph containing $\Sigma^{\ker}$, so in particular $\zeta_1^\circ,\ldots ,\zeta_r^\circ \in \widetilde\Sigma$.  For all $i = 1, \dots , r$, since $\widetilde\Sigma$ is closed, connected, and contains $\zeta_i^\circ$, there exists $\zeta_i' \in [\zeta_i, \zeta_i^\circ]$ such that 
            \begin{equation} \label{piedinodisigmatilde}
                [\zeta_i, \zeta_i^\circ] \cap \widetilde \Sigma = [\zeta_i', \zeta_i^\circ].
            \end{equation}
            For distinct $i,j$, \cref{approxzetacirco} (a) gives us that $[\zeta_i, \zeta_j] = [\zeta_i, \zeta_i^\circ) \sqcup [\zeta_i^\circ, \zeta_j^\circ ] \sqcup (\zeta_j^\circ, \zeta_j]$, which together with \cref{piedinodisigmatilde} yields $[\zeta_i,\zeta_j]\cap \widetilde \Sigma = [\zeta_i',\zeta_i^\circ) \sqcup [\zeta_i^\circ,\zeta_j^\circ] \sqcup (\zeta_j^\circ, \zeta_j'] = [\zeta_i',\zeta_j']$. It then follows from \cref{spancomeunione} that 
            \begin{equation*}
                \Sigma \cap \widetilde\Sigma = \bigcup_{1 \leq i < j \leq r} [\zeta_i', \zeta_j'] \ = \ \spando \{\zeta_1', \ldots, \zeta_r'\} =: \Sigma'.
            \end{equation*}
            Since $\Sigma'\subset \widetilde\Sigma$, we have that $\Sigma'$ is a finite subgraph.
            
            We now define the measure $\tau_R^\alpha$ from the statement. (As mentioned above we will be omitting its dependence from $R$ and $\alpha$ in this proof.)
            \begin{equation}\label{deftau}
                \tau := \tau_R^\alpha := (\deg R) \times \delta_\alpha \ +  \sum_{\zeta\in \pwanb} \left ( t(\dir{\zeta\alpha} ) \  - \sum_{v \in T_\zeta\pwanb \smallsetminus \{\dir{\zeta\alpha}\}} t(v)\right ) \delta_\zeta.
            \end{equation}
            We see that $\tau$ is a well-defined positive measure. Indeed, by \cref{decrescentina}, all coefficients in the sum over $\zeta\in\pwanb$ are finite and nonnegative, except at most for $\zeta=\alpha$, in which case the coefficient is given by \cref{bucogauss} and is bounded from below by $-\deg R$. A priori though, $\tau$ might not be finite or even $\sigma$-finite.\\
            By \cref{piattosuiramoscelli} and \cref{zetaconipuntanoalpha} we  have
            \begin{flalign}\label{tausuiramoscelli}
                &&\sum_{\zeta \in [\zeta_i,\zeta_i']} \left ( t(\dir{\zeta\alpha}) \ - \sum_{v \in T_\zeta\Sigma \smallsetminus \{\dir{\zeta\alpha}\}} t(v) \right ) \delta_\zeta \ = t(\dir{\zeta_i\alpha} ) \times \delta_{\zeta_i}, && \forall \ i \in \{1, \ldots , r \}. 
            \end{flalign}

            By \cref{zerofuoridasigma}, we can restrict the sums in \cref{deftau} to $\zeta \in \Sigma$ and $v \in T_\zeta\Sigma \smallsetminus \{\dir{\zeta\alpha}\}$. Moreover, since $\Sigma' \subset \Sigma$, by \cref{approxzetacirco} (c) we have that $\Sigma = (\Sigma')^\circ \sqcup \left ( \bigsqcup_{i=1}^r [\zeta_i,\zeta_i'] \right )$, and by (b) in the same lemma we have $T_\zeta\Sigma' = T_\zeta\Sigma$, for all $\zeta\in (\Sigma')^\circ$. Using these facts together with \cref{tausuiramoscelli}, we obtain
            \begin{equation} \label{taupiallata}
                \tau - (\deg R ) \times \delta_\alpha = \sum_{\zeta \in (\Sigma')^\circ} \left ( t(\dir{\zeta\alpha} ) \  - \sum_{v \in T_\zeta\Sigma' \smallsetminus \{\dir{\zeta\alpha}\}} t(v)\right ) \delta_\zeta\ + \sum_{i=1}^r t(\dir{\zeta_i\alpha} ) \times \delta_{\zeta_i}.
            \end{equation}
            Let us now compute the Laplacian of $T$ on $\widetilde\Sigma$. Applying the Fundamental Theorem of Calculus to the definition \cref{bigt}, we obtain that
            \begin{align} \label{traducidtatpiccoloversoalpha}
                d_{\dir{\zeta\alpha}}T &= -t(\dir{\zeta\alpha}), \quad &\forall \ \zeta \in \widetilde\Sigma \smallsetminus \{\alpha\},\\
                \label{traducitpiccoloaltrove}
                d_vT &= t(v), \quad &\forall \ v \in T_\zeta\widetilde \Sigma \smallsetminus \{\dir{\zeta\alpha}\}, \quad \forall \ \zeta \in \widetilde\Sigma . 
            \end{align}
            Then, by summing over the directions at $\zeta$ in $\widetilde\Sigma$, we get
            \begin{flalign}\label{tgrandeatpiccolo}
                && -\sum_{v \in T_\zeta\widetilde\Sigma} d_v T \ = \ t(\dir{\zeta\alpha}) \ - \sum_{v \in T_\zeta \widetilde\Sigma \smallsetminus \{\dir{\zeta\alpha}\}} t(v) \, , && \forall \ \zeta \in \widetilde\Sigma .
            \end{flalign}
            By \cref{contaccio} the function $t$ is constant on segments short enough stemming from a point, so it follows from \cref{traducidtatpiccoloversoalpha} and \cref{traducitpiccoloaltrove} that 
            \begin{equation}\label{tincippiasigmatilde}
                T \in \cpa(\widetilde\Sigma).
            \end{equation}
            We can then plug \cref{tgrandeatpiccolo} into the formula \cref{laplacianopercippiasigma}, thus obtaining
            \begin{equation}\label{laplacianoditsusigmatilde}
                \Delta_{\widetilde\Sigma} T = \sum_{\zeta \in \widetilde \Sigma} \left ( t(\dir{\zeta\alpha}) \ - \sum_{v \in T_\zeta\widetilde\Sigma \smallsetminus \{\dir{\zeta\alpha}\}} t(v) \right ) \delta_\zeta \, .
            \end{equation}
            By \cref{zerofuoridasigma} we can restrict the outer sum to indices $\zeta \in \Sigma \cap \widetilde \Sigma = \Sigma'$ and the inner sum to indices $v \in (T_\zeta\Sigma \cap T_\zeta\widetilde\Sigma ) \smallsetminus \{\dir{\zeta\alpha}\}$, which by \cref{intersezionedispazitangentiinterni} can be written as $v \in T_\zeta\Sigma' \smallsetminus \{\dir{\zeta\alpha}\}$.  Notice that $\Sigma^{\ker} \subset \Sigma \cap \widetilde\Sigma = \Sigma'$. It then follows from \cref{alphastainsigmacare} and \cref{monotoniainteriore} that $\alpha \in (\Sigma')^\circ$, hence $\dir{\zeta_i'\alpha} \in T_{\zeta_i'}\Sigma'$ for all $i=1, \ldots,r$. Using (b) from \cref{approxzetacirco}, we find that in fact $\{\dir{\zeta_i'\alpha}\}  = T_{\zeta_i'} \Sigma'$, for all $i=1,\ldots,r$. By these observations we can rewrite \cref{laplacianoditsusigmatilde} as
            \begin{equation}\label{deltasigmatildefinale}
                \Delta_{\widetilde\Sigma}T = \sum_{\zeta\in (\Sigma')^\circ}  \left ( t(\dir{\zeta\alpha}) \ - \sum_{v \in T_\zeta\Sigma' \smallsetminus \{\dir{\zeta\alpha}\}} t(v) \right ) \delta_\zeta \ + \sum_{i=1}^r t(\dir{\zeta_i'\alpha}) \times \delta_{\zeta_i'}.
            \end{equation}
            In the following display the first equality follows from \cref{piedinodisigmatilde}. The second one follows from the fact that $\zeta_i'\in [\zeta_i,\zeta_i^\circ]$ for all $i$ and  \cref{tcostantesuiramoscelli}.
            \begin{flalign*}
                &&r_{\widetilde \Sigma}(\zeta_i)= \zeta_i', \qquad t(\dir{\zeta_i\alpha}) = t(\dir{\zeta_i'\alpha}), && \forall \ i \in \{1,\ldots, r\} \minis .
            \end{flalign*}
            Now, $\Sigma'\subset \widetilde \Sigma$, so $r_{\widetilde\Sigma}|_{\Sigma'} = \id_{\Sigma'}$. Considering this and the display above, if we evaluate $(r_{\widetilde\Sigma})_*$ on \cref{taupiallata}, we obtain the right-hand side in \cref{deltasigmatildefinale}, and therefore
            \begin{equation}\label{tausiritraeadeltasigmatildet}
                (r_{\widetilde\Sigma})_* \left (\tau - (\deg R ) \times \delta_\alpha \right ) = \Delta_{\widetilde\Sigma} T .
            \end{equation}
            Using \cref{decrescentina} in \cref{laplacianoditsusigmatilde}, we see that support of the negative part of $\Delta_{\widetilde\Sigma}T$ is contained in $\{\alpha\}$. Since $\Delta_{\widetilde{\Sigma}}T \in \meas_0 (\widetilde\Sigma)$, the positive and negative parts of $\Delta_{\widetilde\Sigma}$ have the same total mass, hence
            \begin{equation*}
                |\Delta_{\widetilde\Sigma}T|(\widetilde\Sigma) = -2 \Delta_{\widetilde\Sigma}T(\{\alpha\}) \leq 2\deg R \minis ,
            \end{equation*}
            where the inequality follows from \cref{bucogauss} and \cref{traducitpiccoloaltrove}. Since $\widetilde\Sigma$ is an arbitrary finite subgraph containing $\Sigma^{\ker}$, by \cref{nonservonotuttiigrafi} the inequality above and \cref{tincippiasigmatilde} are enough to conclude that $T \in \bdv(\pwanb)$. 
            
            Taking $\widetilde\Sigma = \Sigma^{\ker}$ to avert empty hypotheses, the equality \cref{tausiritraeadeltasigmatildet} shows that $\tau$ has total mass $\deg R$, because the total mass $0 = \Delta_{\Sigma^{\ker}} T(\Sigma^{\ker})$ is preserved under the pushforward operator $(r_{\Sigma^{\ker}})_*$. In \cref{deftau} we see that $\tau$ is a linear combination of Dirac deltas with integer coefficients. The finite total mass of $\tau$ then implies that the sum in \cref{deftau} is actually finite, so that in particular we obtain that $\tau$ is a finite Radon measure, that is, $\tau \in \meas(\pwanb)$. Since moreover we established that $\tau (\pwanb) = \deg R$, it follows that $\tau - (\deg R) \times \delta_\alpha \in \meas_0(\pwanb)$. Now, in the equality \cref{tausiritraeadeltasigmatildet}, $\widetilde\Sigma$ is an arbitrary finite subgraph containing $\Sigma^{\ker}$, therefore by \cref{nonservonotuttiigrafi} we conclude that 
            \begin{equation}\label{deltatugualetau}
                \Delta T = \tau - (\deg R) \times \delta_\alpha.
            \end{equation}
            Now  plugging \cref{piattofissoatt}, \cref{piattononfisso}, \cref{bozzononatt} and \cref{bucogauss} in \cref{deftau}, by the equality above we have
            \begin{equation}\label{ultimodeltat}
                \tau \ 
                = \sum_{\zeta\in \nafl(R) \smallsetminus \{\alpha\}} \left ( \deg_R \zeta - \uuno\left ( \substack{\dir{\zeta\alpha} \text{ is a non-repelling} \\ \text{fixed direction}}\right )\right ) \delta_\zeta \ + \ (\deg_R \alpha) \times \uuno\left ( \alpha \text{ is fixed}\right ) \times \delta_\alpha \, .
            \end{equation}
            Since $\alpha \in \hypb$, it cannot be an attracting fixed point. Therefore, $\uuno \left ( \alpha \text{ is fixed}\right ) = \uuno \left ( \alpha\in \nafl(R) \right )$. Moreover, since $\dir{\alpha\alpha} = \alpha$ is not a direction, we have $\uuno\left ( \substack{\dir{\alpha\alpha} \text{ is a non-repelling} \\ \text{fixed direction}}\right ) = 0$. So we can write
            \[  (\deg_R \alpha) \times \uuno\left ( \alpha \text{ is fixed}\right ) = \left ( \deg_R \alpha - \uuno\left ( \substack{\dir{\alpha\alpha} \text{ is a non-repelling} \\ \text{fixed direction}}\right )\right ) \times \uuno 
            \left ( \alpha \in \nafl(R) \right ) ,  \]
            which plugged into \cref{ultimodeltat} yields
            \[  \tau = {\mspace{-8mu} }\sum_{\zeta \in \nafl(R)} {\mspace{-8mu} } \left ( \deg_R \zeta - \uuno \left ( \substack{ \dir{\zeta \alpha} \text{ is a non-repelling}  \\ \text{fixed direction}} \right ) \right ) \delta_\zeta = {\mspace{-8mu} } \sum_{\Gamma \in \nas(R)} \sum_{\zeta \in \Gamma} \left ( \deg_R \zeta - \uuno \left ( \substack{ \dir{\zeta \alpha} \text{ is a non-repelling}  \\ \text{fixed direction}} \right ) \right ) \delta_\zeta \minis .\]
            The equalities \cref{laplaciazzoT} from the statement are now given by rearranging  \cref{deltatugualetau} and by applying Lemma \ref{piuvicinoadalphaspento} (a), (b) in the equality above.
        \end{proof}
        \begin{cor}\label[cor]{tauincasobenridotto}
            Let $R \in K(z)$ be a nonconstant rational function with potential good reduction.\\
            Then $\nas(R)$ has exactly one element $\Gamma$. Moreover, if $\alpha \in \hypb$, then:
            \begin{enumerate}[(a)]
                \item if $\deg R = 1$, then $\tau_R^\alpha = (r_\Gamma)_*\delta_{\alpha}$;
                \item if $\deg R \geq 2$ and we denote  $\{\xi\} = \julia_R$ (see \cref{reduciergodo}),      then $\tau_R^\alpha = (\deg R -1) \times \delta_\xi + (r_\Gamma)_*\delta_{\alpha}$.
            \end{enumerate}
        \end{cor}
        \begin{proof}
            Let $R \in K(z)$ with $\deg R = 1$. By \cref{naflnonvuoto} there exists $\Gamma \in \nas(R)$. Since $\deg_R \Gamma \geq 1$, by \cref{tausuigamma} we conclude that $\Gamma$ is the only non-attracting subtree of $R$. \\
            Let now $\deg R \geq 2$. Then by \cref{reduciergodo} we have that $\julia_R = \{\xi\}$ for some type II point $\xi \in \hypb$. Since $R^{-1}\julia_R = \julia_R$, it follows from \cref{degreessumup} (b) that $\xi$ is a repelling fixed point with $\deg_R \xi = \deg R$. Let $\Gamma \in \nas(R)$ be the non-attracting subtree containing $\xi$. By the definition \cref{defgradociocco} and by \cref{tausuigamma}, we have that 
            \begin{equation}\label{puntogiuliotto}
                \deg R \geq \deg_R \Gamma \geq  1 + \deg_R \xi -1 = \deg R.
            \end{equation}
            Invoking \cref{tausuigamma} again, it follows that $\Gamma$ is the only non-attracting subtree of $R$.\\
            Let now $R\in K(z)$ have potential good reduction and any degree. By what we showed above, we can write $\nas(R) =\{\Gamma\}$, and therefore by \cref{laplaciazzoT},
            \begin{equation}\label{tausuunsologamma}
                \tau_R^\alpha = (r_\Gamma)_*\delta_\alpha + \sum_{\zeta\in \Gamma} \left ( \deg_R \zeta - 1 \right ) \times \delta_\zeta .
            \end{equation}
            If now $\deg R = 1$, then $\deg_R \zeta = 1$ for all $\zeta \in \pwanb$.  If instead $\deg R \geq 2$, then it follows from \cref{puntogiuliotto} and the definition of degree of a non-attracting subtree \cref{defgradociocco}, that $\xi$ is the only point in $\Gamma$ satisfying $\deg_R \xi \geq 2$, and in fact $\deg_R \xi = \deg R$. Applying these considerations in \cref{tausuunsologamma} yields (a) and (b).
        \end{proof}
        The condition that $\nas(R)$ only have one element is not sufficient for $R$ to have potential good reduction, as we see in the following example.
        \begin{ex} \label[ex]{esempiodirumely}
            In \cite{rumely2014geometry} \S 11 Example C, for each integer $d\geq 2$, Rumely gives an example of a rational function of degree $d$ with $d-1$ type II repelling fixed points in $\hypb$, all with local degree $2$. Let $R \in K(z)$ be such a function. Then considering \cref{tausuigamma} and \cref{gradiinnafl} (b), we have
            \begin{multline*}
                d = \sum_{\Gamma \in \nas(R)} \left ( 1  + \sum_{\zeta \in \Gamma} (\deg_R \zeta - 1 ) \right ) = \# \nas(R) + \sum_{\zeta \in \rf(R) \cap \hypb} (2 -1 ) \\
                = \# \nas(R) + d - 1 \geq 1 + d  - 1 = d 
            \end{multline*}
            It follows that $\# \nas(R) = 1$. Since there are $d-1$ repelling fixed points of type II and therefore at least $d-1$  points in $\julia_R = \supp \nu_R$, for $d \geq 3$ the rational function $R$ cannot have potential good reduction by \cref{reduciergodo}. 
        \end{ex}
        \begin{prop}\label[prop]{formaditausuciascungamma}
        	Let $\Gamma \in \nas (R)$.
        	\begin{enumerate}[(a)]
        		\item If $\Gamma$ contains a type I repelling point $\zeta$, then 
        		\begin{equation*}
        			\Gamma = \{\zeta\}, \quad \tau^\alpha_\Gamma = \delta_\zeta, \ \ \ \text{and}  \  \deg_R \Gamma = \deg_R \zeta = 1 \minis .
        		\end{equation*}
        		\item If $\Gamma$ is a repelling subtree containing no repelling type I points, then
        		\begin{equation*}
        			\tau^\alpha_R|_{\Gamma} = (r_\Gamma)_*\delta_\alpha + \sum_{\zeta \in \Gamma \cap \fix_{>1}(R)} ( \deg_R \zeta -1 ) \delta_\zeta \minis .
        		\end{equation*} 
        		\item If $\Gamma$ is an indifferent subtree, then \begin{equation*}
        			\tau^\alpha_R|_{\Gamma} = (r_\Gamma)_*\delta_\alpha \minis .
        		\end{equation*}
        	\end{enumerate}
        \end{prop}
        \begin{proof}
        	Suppose $\Gamma \in \nas(R)$ and that $\zeta\in \Gamma$ is a repelling type I point. Then by \cref{classifitangent} (a) it only has one direction and by \cref{taledirezionetalepunto} (b)  that direction is repelling. It follows that every nonempty segment $(\zeta,\xi)$ will contain some nonempty subsegment $(\zeta,\zeta')$ containing no fixed points. We conclude that $\Gamma = \{\zeta\}$. By \cref{classifidinamica} (d), we have $\deg_R\zeta = 1$. Now the rest of point (a) follows from \cref{tausuigamma}.
        	Points (b) and (c) follow from \cref{tausuigamma} and \cref{gradiinnafl} (a).
        \end{proof}

        \section{Equidistribution}\label{sezequidistro}
        \subsection{Equidistribution of the measures $\tau_{R^n}^\alpha$}\label{sezequidistrotau}
        For a Radon probability measure $\mu$ on $\pwanb$, we define
        \begin{flalign*}
            && f_\mu^\alpha(\alpha) = 0, \quad f_\mu^\alpha(\xi) = \mu(V^\alpha(\xi)) , \qquad F_\mu^\alpha(\zeta) = \int_\alpha^\zeta f_\mu^\alpha d \lambda && \forall \ \xi \in \hypb \smallsetminus \{\alpha\}, \ \zeta \in \hypb .
        \end{flalign*}
        Notice that by the definition of $V^\alpha(\xi)$ and by \cref{lemmabruttu},  $f^\alpha_\mu$ is monotone on $(\alpha,\zeta)$ for all $\zeta\in \pwanb$, so  $F^\alpha_\mu$ is well-defined.
        \begin{prop}\label[prop]{potenzialiefzdidistro}
            Let $\mu \in \meas(\pwanb)$ be a discrete measure on $\pwanb$. Then
            \begin{equation}\label{fgrandeeilpotenziale}
                F^\alpha_\mu \in \bdv(\pwanb), \qquad \Delta F^\alpha_\mu = \mu - \delta_\alpha \minis .
            \end{equation}
            Moreover, if $\mu$ is a discrete \textit{probability} measure, then $-F^\alpha_\mu \in \ag[\alpha]$.
        \end{prop}
        \begin{proof}
            Suppose first that $\zeta\in \hypb$ and $\mu= \delta_\zeta$. Then we can verify that $f_\mu^\alpha = \uuno_{(\alpha,\zeta)}$ and that on each finite subgraph $\Sigma$ containing $[\zeta,\alpha]$ we have $\Delta_\Sigma F^\alpha_\mu = \delta_\zeta - \delta_\alpha$, which shows in particular that $|\Delta_\Sigma F^\alpha_\mu|(\Sigma) \leq 2 $. It follows in particular that $F^\alpha_\mu \in \bdv(\pwanb)$ and that $\Delta F^\alpha_\mu = \mu - \delta_\alpha$. Notice that $F^\alpha_\mu = F \circ r_{[\alpha,\zeta]}$, where $F : [\alpha,\zeta] \to \RR, \xi \mapsto \lambda((\alpha,\xi)) =\varrho(\alpha,\xi)$. It follows from this that $F^\alpha_\mu$ is continuous, so we obtain that $-F^\alpha_\mu \in \ag[\alpha]$.

            Let now $\zeta\in \pwanb$ be arbitrary and again $\mu= \delta_\zeta$. Excluding the trivial case where $\zeta=\alpha$, we can choose a sequence $\zeta_n \in (\alpha,\zeta)$ that is increasing to $\zeta$ if we order the segment from $\alpha$ to $\zeta$. Then if $\mu_n = \delta_{\zeta_n}$, the functions $f^\alpha_{\mu_n} = \uuno_{(\alpha,\zeta_n)}$, in turn, increase to $f^\alpha_{\mu} = \uuno_{(\alpha,\zeta)}$ and by monotone convergence we have $F^\alpha_\mu = \lim_{n\to\infty} F^\alpha_{\mu_n}$ pointwise. Since $-F_{\mu_n}^\alpha \in \ag[\alpha]$ for all $n$, it follows from \cref{arachelloverdechiuso} that $-F^\alpha_\mu \in \ag[\alpha]$ and from the continuity of $-\Delta$ in \cref{arachelloverdesattoh} that $\Delta F^\alpha_\mu = \lim_{n\to\infty} \mu_n - \delta_\alpha = \mu -\delta_\alpha$. Here the last equality follows by directly computing that for a continuous function $f: \pwanb \to \RR$ we have $\int fd \mu_n = f(\zeta_n) -f(\alpha) \longrightarrow f(\alpha) - f(\zeta)$, as $n\to\infty$. We have therefore obtained \cref{fgrandeeilpotenziale} for all $\mu = \delta_\zeta, \ \zeta\in \pwanb$.
            
            The case of general finite discrete measures follows by linearity. 
            
            If now $\mu$ is a discrete probability measure, it is in particular a linear combination of Dirac masses $\delta_\zeta$, for which we showed that $-F^\alpha_{\delta_\zeta} \in \ag[\alpha]$. We then obtain by linearity that $-F_\mu^\alpha$ is in $\bdv(\pwanb)$, that it is upper semi-continuous on $\hypb \smallsetminus \{\alpha\}$, and since $\mu$ is a probability measure, $\Delta ( - F_\mu^\alpha ) = \delta_\alpha -\mu \in  \meas^+_1[\alpha]$. Therefore $-F^\alpha_\mu \in \ag[\alpha]$.
        \end{proof}
        \begin{theorem}[manual-num=\ref{equidistrotau}]
            Let $R\in K(z)$ with $\deg R \geq 2$ and let $\alpha \in \hypb$. As $n \to \infty$, the probability measures $(\deg R^n)^{-1}\tau_{R^n}^\alpha$ converge to the equilibrium measure $\nu_R$.
        \end{theorem}
        \begin{proof}
            In the case of potential good reduction, by \cref{reduciergodo} there exists a type II point $\xi \in \hypb$ such that $\nu_R = \delta_\xi$, so that the convergence in the statement follows readily from  \cref{tauincasobenridotto} (b).
            
            Assume now that $R$ has no potential good reduction. To simplify our notation, we will write $t_{R^n}^\alpha =: t_n, T_{R^n}^\alpha =: T_n, \tau_{R^n}^\alpha =: \tau_n$, omitting the dependence from $R$ and $\alpha$.\\
            By Brolin's \cref{brolin}, since  $\zeta=\zeta(0,1) \in \hypb$ is outside of the exceptional set of $R$, we have a sequence of probability measures $\nu_n = (\deg R^n)^{-1}(R^n)^* \delta_\zeta$ converging to $\nu_R$. Since $R$ has no potential good reduction, by \cref{reduciergodo} we know that $\nu_R$ does not charge points. Therefore, for all $\xi \in \hypb$ we have $\nu_R(\partial V^\alpha(\xi)) = \nu_R(\{\xi\}) = 0$. It then follows from the Portmanteau \cref{portmantone} that
            \begin{equation*}
                f^\alpha_{\nu_n}(\xi) = \nu_n( V^\alpha(\xi))  \xrightarrow[ \quad n \to \infty \quad]{ } \nu_R(V^\alpha(\xi)) = f^\alpha_{\nu_R}(\xi) \minis .
            \end{equation*}
            Since $0 \leq f^\alpha_{\nu_n} \leq 1$ for all $n$, by dominated convergence we obtain
            \begin{flalign*}
                && F^\alpha_{\nu_n}(\zeta) \xrightarrow[\quad n\to\infty \quad ]{} F^\alpha_{\nu_R}(\zeta) , && \forall \ \zeta \in \hypb .
            \end{flalign*}
            Each $\nu_n$ is a discrete probability measure, so by \cref{potenzialiefzdidistro} we have $-F^\alpha_{\nu_n} \in \ag[\alpha]$ and $\Delta F^\alpha_{\nu_n} = \nu_n - \delta_\alpha$, for all $n \in \NN$. By \cref{arachelloverdechiuso} we have $-F^\alpha_{\nu_R} \in \ag[\alpha]$ and by continuity of $-\Delta$ in \cref{arachelloverdesattoh}, we have
            \begin{equation} \label{deltaeffemu}
                 \Delta F^\alpha_{\nu_R} = \lim_{n \to \infty} \Delta F^\alpha_{\nu_n} = \lim_{n \to \infty } \nu_n - \delta_\alpha  = \nu_R - \delta_\alpha \minis . 
            \end{equation}
            Again using Brolin's \cref{brolin} and the Portmanteau \cref{portmantone}, we have that
            \begin{flalign}\label{convergiotpiccolo}
                &&(\deg R^n)^{-1} t_n(\xi) = (\deg R^n)^{-1}(R^n)^*\delta_\xi (V^\alpha(\xi)) \xrightarrow[\quad n \to \infty \quad]{} \nu_R(V^\alpha(\xi)) = f^\alpha_{\nu_R}(\xi),  && \forall \ \xi \in \hypb .
            \end{flalign}
            By \cref{tausuigamma} the measure $(\deg R^n)^{-1}\tau_n$ is a discrete probability measure for all $n$, whence, by \cref{potenzialiefzdidistro}, we have   
            \begin{equation}\label{deltaeffetauenne}
                -(\deg R^n)^{-1}F^\alpha_{\tau_n} \in \ag[\alpha] \minis , \qquad \Delta F^\alpha_{\tau_n} = \tau_n - (\deg R^n) \times \delta_\alpha = \Delta T_n \minis ,
            \end{equation} 
            where the second equality is the one from \cref{laplaciazzoT}.
            Hence $F^\alpha_{\tau_n}$ and $T_n$ must differ by a constant, whence $-(\deg R^n)^{-1} T_n \in \ag[\alpha]$.
            Using \cref{convergiotpiccolo}, and again dominated convergence, we obtain that
            \begin{equation*}
                 (\deg R^n)^{-1}T_n(\zeta) = (\deg R^n)^{-1}\int_\alpha^\zeta t_n(\xi) d \lambda (\xi)  \xrightarrow[\quad n \to \infty \quad ]{} \int_\alpha^\zeta f_{\nu_R} (\xi) d \lambda (\xi) = F^\alpha_{\nu_R}(\zeta) , \qquad \forall \ \zeta \in \hypb.
            \end{equation*}
            It then follows from \cref{deltaeffemu} and continuity of $-\Delta$ in \cref{arachelloverdesattoh} that $(\deg R^n)^{-1} \tau_n  \xrightarrow[\ n \to \infty \ ]{} \nu_R$.
        \end{proof}
    \subsection{Equidistribution of repelling points}\label{sezequidistropuntirep}
    Recall that we denote by $\rf(R)$ the set of its repelling fixed points of $R\in K(z)$ and we set
    \begin{equation*}
        \sigma_R \ = \sum_{\zeta \in \rf(R)} (\deg_R \zeta ) \times \delta_\zeta \minis .
    \end{equation*}
    \paragraph{Polynomials.} We will now show that when $R\in K(z)$ is a polynomial, for $\alpha\in \hypb$ close to $\infty \in \pwanb$, we have $\tau_R^\alpha = \sigma_R$. We first prove the following well-known fact, for lack of a suitable reference.
    \begin{lemma}\label[lemma]{polifissanoinfinititangenti}
        Let $R\in K[z]$ be a nonconstant polynomial, $\zeta\in \pwanb, v \in T_\zeta\pwanb$. Under the map on tangent spaces $R: T_\zeta\pwanb \to T_{R(\zeta)}\pwanb$, the direction $\dir{\zeta\infty}$ is the only preimage of the direction $\dir{R(\zeta)\infty}$.  
    \end{lemma}
    \begin{proof}
        Suppose that $v\in T_\zeta\pwanb$ is mapped to $\dir{R(\zeta) \infty} \in T_{R(\zeta)}\pwanb$. By \cref{imgofdisk}, the disk $\bdisk(v)$ contains a preimage of $\infty$. Since $R$ is a polynomial, this preimage must be $\infty$ and therefore $v = \dir{\zeta\infty}$. On the other hand, by \cref{degreessumup} (a), the map on tangent spaces is surjective, so in particular, $\dir{R(\zeta)\infty}$ has a preimage.
    \end{proof}

    \begin{theorem}[manual-num=\ref{polynomials}]
         If $R \in K[z]$ is a polynomial with $\deg R \geq 2$ and $\alpha$ is contained in the Fatou component of $\infty \in \pwanb$, then $\sigma_R = \tau_R^\alpha$. In particular, $\deg R = \sum_{\zeta \in \rf(R)} \deg_R \zeta$, and if $\deg R \geq 2$, we have the equidistribution of repelling points:
        \[\left (\deg R^n \right )^{-1}  \sigma_{R^n} \; \xrightarrow[\quad n\to \infty \quad ]{} \; \nu_R \minis .\]
    \end{theorem}
    \begin{proof}
        Notice that since $R$ is a polynomial, $R^{-1}(\infty) = \{\infty\}$.
        Therefore, $\deg_R \infty = \deg R \geq 2$, which makes $\infty$ an attracting fixed point by \cref{classifidinamica} (a). In particular, $\infty \notin \nafl(R)$. Let $\Gamma \in \nas(R)$.
        
        Since $\infty \notin \nafl(R)$, we have $r_\Gamma(\infty) \neq \infty$ and therefore $\dir{r_\Gamma(\infty) \infty}$ is a direction. By \cref{polifissanoinfinititangenti} it is fixed.
        Taking $\alpha = \infty$ in \cref{piuvicinoadalphaspento} (a) and using the dichotomy \cref{dicotomiadirezioni} (a), we find that $\dir{r_\Gamma(\infty) \infty}$ must be a repelling fixed direction, which makes $r_\Gamma(\infty)$ a repelling fixed point by \cref{taledirezionetalepunto} (b).
      
        Let now $\zeta \in \Gamma \smallsetminus \{r_\Gamma(\infty)\}$. Then $\dir{\zeta\infty} \in T_\zeta\Gamma$, so by \cref{indifferelligammelli} it is an indifferent fixed direction, in particular, not a repelling fixed direction. Then if $\zeta$ is type I, by \cref{classifitangent} (a) and \cref{taledirezionetalepunto} (b), we see that $\zeta$ cannot be a repelling fixed point. Then since $\zeta\in \nafl(R)$, by \cref{gradiinnafl} (b), we have $\deg_R \zeta =1$. Suppose now instead that $\zeta\in \hypb$. Since $\dir{\zeta\infty}$ is an indifferent fixed direction, by \cref{indifferentdirections} (b), we have $\deg_R \dir{\zeta\infty} = 1$. It then follows from \cref{polifissanoinfinititangenti} and \cref{degreessumup} (a) that $\deg_R \zeta = 1$, and therefore from the definitions in Section \ref{defqualitadinamiche} that $\zeta$ is not a repelling fixed point.
        
        Summarizing the last two paragraphs:  the point $r_\Gamma(\infty)\in \Gamma$ is a repelling fixed point in $\Gamma$ and is in fact the unique repelling point in $\Gamma$. We conclude that, since $\Gamma \in \nas(R)$ was arbitrary,
        \begin{flalign}\label{tausugammapoli}
            &&(r_\Gamma)_*\delta_\infty + \sum_{\zeta\in \Gamma} (\deg_R \zeta - 1) \times \delta_\zeta = (\deg r_\Gamma(\infty)) \times (r_\Gamma)_* \delta_\infty \minis , &&  \forall \ \Gamma \in \nas(R). 
        \end{flalign}
        Let $\Gamma \in \nas(R)$. Since $\alpha$ is in the Fatou component containing $\infty$, and the repelling point $r_\Gamma(\infty)$ is contained in $\julia_R$ by \cref{chistadove} (b), it follows that $\bdisk(\dir{r_\Gamma(\infty)\infty}) = \bdisk( \dir{r_\Gamma(\infty) \alpha})$. 
        From \cref{retrazionestessazione} it then follows that $r_\Gamma(\alpha) = r_\Gamma(\infty)$.
        By definition of $r_\Gamma: \pwanb \to \Gamma$, for all $\zeta\in \Gamma\smallsetminus \{r_\Gamma(\infty)\}$, we have that $\dir{\zeta\infty}= \dir{\zeta r_\Gamma(\infty)} = \dir{\zeta r_\Gamma(\alpha)} = \dir{\zeta \alpha}$.
        Moreover, all repelling fixed points are contained in some $\Gamma \in \nas(R)$. Hence, by \cref{laplaciazzoT} and \cref{tausugammapoli},
        \begin{equation}\label{nuesigma}
            \tau_R^\alpha =  \sum_{\zeta \in \rf(R)} (\deg_R \zeta ) \times \delta_\zeta  = \sigma_R.
        \end{equation}
        In particular we have that $\deg R = \tau_R^\alpha(\pwanb) = \sigma_R(\pwanb) = \sum_{\zeta \in \rf(R)} \deg_R \zeta$. By applying the equality \cref{nuesigma} to the iterates of $R$, we obtain the equidistribution of repelling points from \cref{equidistrotau}.
    \end{proof}
    \paragraph{Rational functions with connected Julia set.} Recall the following measures mentioned in the introduction:
    \begin{equation}\label{errorucci}
        \tau^\alpha_{R, \operatorname{ind}} =\sum_{\substack{\Gamma \in\nas(R) \\ : \ \Gamma \cap \rf(R) = \varnothing }} \delta_{r_\Gamma(\alpha)} \minis , \qquad 
        \tau^\alpha_{R,\operatorname{sh}} = \sum_{\substack{\Gamma \in \nas(R) \\ : \ \Gamma \cap \rf(R) \neq \varnothing}} \left (  - \minis\delta_{r_\Gamma(\alpha)} +  \sum_{\zeta\in \rf(R) \cap\Gamma} \delta_\zeta \right ).
    \end{equation}
    It follows from \cref{formaditausuciascungamma} that
    \begin{equation*}
        \tau_R^\alpha - \sigma_R = \tau^\alpha_{R,\operatorname{ind}} - \tau^\alpha_{R,\operatorname{sh}} \minis .
    \end{equation*}
    Notice that both $ \tau^\alpha_{R,\operatorname{ind}}$ and $\tau^\alpha_{R,\operatorname{sh}}$ are nonnegative, and therefore are respectively the positive and negative part of $\tau_R^\alpha-\sigma$.
    
    We will now show that when the Julia set of $R$ is connected, the error term $\tau^\alpha_{R^n, \operatorname{sh}}$ vanishes for all $n$ if we choose $\alpha \in \julia_R$, so that repelling points equidistribute if we assume the amount of indifferent subtrees to become negligible in the limit. The key observation is that when $\tau^\alpha_{R, \operatorname{sh}} \neq 0$, there is a segment of the Fatou set between two fixed repelling points.
    \begin{lemma}\label[lemma]{seriddiallorafatu}
        Let $R\in K(z)$ with $\deg R\geq 2$. Suppose $\zeta,\zeta'\in \pwanb$ are distinct and that $R|_{(\zeta,\zeta')} = \id_{(\zeta,\zeta')}$. Then $(\zeta,\zeta') \cap \fatou_R$ contains a non-empty open segment.
    \end{lemma}
    \begin{proof}
        Up to restricting ourselves to a smaller segment, by \cref{segmentiapertiduetre} we can assume that $[\zeta,\zeta']$ consists of type II and type III points, with $\zeta,\zeta'$ type II. Up to a change of coordinates, we can then assume $\zeta= \zeta(0,r), \zeta'=\zeta(0,r')$, with $r'>r>0$. Up to replacing $\zeta'$ again, we can assume $r'>0$ close enough to $r>0$ so as to have no poles in $\bannulus(\zeta,\zeta')$. Let now $\xi \in (\zeta,\zeta')$. For all $v \in T_\xi\pwanb\smallsetminus \{\dir{\xi\zeta}, \dir{\xi\zeta'}\}$ we have $R(\bdisk(v)) \subset R(\bannulus(\zeta,\zeta') ) \subset \pwanb \smallsetminus \{\infty \}$. On the other hand, $\dir{\xi\zeta}, \dir{\xi\zeta'}$ are fixed by $R:T_\xi\pwanb \to T_\xi\pwanb$. By \cref{chistadove} (a), (c) it then follows that $\xi \in \fatou_R$. Since $\xi\in (\zeta,\zeta')$ was arbitary, we obtain that $(\zeta,\zeta')\subset \fatou_R$.
    \end{proof}
   
    \begin{theorem}[manual-num=\ref{connectedcounted}]
         Let $R\in K(z)$ with $\deg R \geq 2$ and suppose that $\julia_R$ is connected. Then $\julia_R \cap \hypb \neq \varnothing$, and if $\alpha \in \julia_R \cap \hypb$, then $\tau^\alpha_{R,\operatorname{sh}} = 0$.\\
         Therefore, picking $\alpha \in \julia_R \cap \hypb$ and assuming additionally that $\sigma_{R^n}(\pwanb) \sim \deg R^n$, or equivalently that $\tau^\alpha_{R^n,\operatorname{ind}}(\pwanb) = o(\deg R^n)$, we have 
        \begin{equation*}
            \left(\deg R^n \right )^{-1}\sigma_{R^n} \xrightarrow[\; n\to \infty \;]{} \nu_R \minis .
        \end{equation*}
    \end{theorem}
    \begin{proof}
        Since $\julia_R$ is connected, by \cref{reduciergodo}, it either consists of one point in $\hypb$ or it contains a non empty open segment. Therefore $\julia_R \cap \hypb \neq \varnothing$, so we can fix an $\alpha \in \julia_R \cap \hypb$.
        
        Let $\Gamma$ be a non-attracting subtree containing a repelling point $\zeta$. Since $\julia_R$ is connected, we have $(\zeta,\alpha) \subset \julia_R$. It then follows from \cref{seriddiallorafatu} that $R$ cannot restrict to the identity on any subsegment $(\zeta,\zeta') \subset (\zeta,\alpha)$. Therefore $(\zeta,\alpha) \cap \Gamma = \varnothing$, from which we obtain that $\zeta= r_\Gamma(\alpha)$.
        We have thus shown that if a non-attracting subtree $\Gamma$ contains a repelling point, it is unique and equal to $r_\Gamma(\alpha)$. Looking at \cref{errorucci}, we then see that $\tau^\alpha_{R,\operatorname{sh}} = 0$. Since $\julia_R = \julia_{R^n}$ for all $n \geq 1$, we can apply the same reasoning to all iterates of $R$, thus obtaining $\tau^\alpha_{R^n, \operatorname{sh}} = 0$ for all $n\geq 1$, that is, $\tau_{R^n}^\alpha - \sigma_{R^n} = \tau^\alpha_{R^n,\operatorname{ind}}$. Therefore assuming $\sigma_{R^n}(\pwanb) \sim \deg R^n$ is equivalent to assuming $\tau^\alpha_{R^n,\operatorname{ind}}(\pwanb) = o(\deg R^n )$. With this extra hypotesis, $(\deg R^n)^{-1}\tau^\alpha_{R^n}- (\deg R^n)^{-1}\sigma_{R^n}$ is a sequence of nonnegative measures whose total mass tends to zero as $n\to \infty$, therefore showing that $(\deg R^n)^{-1}\sigma_{R^n}$ has the same limit as $(\deg R^n)^{-1}\tau^\alpha_{R^n}$, which by \cref{equidistrotau} is $\nu_R$.        
    \end{proof}
    \paragraph{Rational functions with positive Lyapunov exponent.}
    Let $R\in K(z)$. Denoting by $||R'|| : \pwanb \to \RR_{\geq 0}$ the spherical derivative of $R$ (see \cite{MR3369346} for its definition), the integral 
    \begin{equation}
        L(R)=\int_\pwanb \log || R'|| d \nu_R
    \end{equation}
    is known as the \textit{Lyapunov exponent} of $R$ (with respect to $\nu_R$).
    \begin{theorem}[manual-num=\ref{positivelyapunov}]
        Let $R\in K(z)$ with $\deg R \geq 2$. If $L(R) >0$, then we have
        \begin{equation}\label{sparisconogliperbolici}
            (\deg R^n)^{-1} \sum_{\zeta \in \fix_{>1} \cap \hypb} (\deg_R \zeta) \times\delta_\zeta \ \xrightarrow[\ n\to\infty \ ]{} \ 0 \minis , \text{ weakly on } \pwanb .
        \end{equation}
    \end{theorem}
    \begin{proof}
        If $\zeta\in \rf(R) \cap \hypb$, then $\deg_{R} \zeta \geq 2$, whence $\deg_{R} \zeta \leq 2 (\deg_{R} \zeta - 1)$. It then follows from the definition \cref{defgradociocco} that
        \begin{multline}\label{ildoppiogrado}
             \sum_{\zeta \in \Gamma \cap \rf(R)} \deg_{R} \zeta   
             \leq \sum_{\zeta \in \Gamma \cap \rf(R) \cap \hypb} 2(\deg_{R} \zeta - 1) 
             \leq \ 2\deg_{R} \Gamma \minis , \\
             \qquad  \forall \ \Gamma \in \nas(R), \ \Gamma \cap \rf(R) \subset \hypb \minis .
        \end{multline} 
        Instead, if $\Gamma$ contains a repelling type I point $\zeta$, by \cref{formaditausuciascungamma} (a) we have $\Gamma = \{\zeta\}$ and $\deg_R \Gamma = 1$. It then follows  from \cref{tausuigamma} that
        \begin{equation}\label{tipounotipo2}
            \deg R \ = \ \# \left (\rf(R) \cap \PP^1(K) \right ) \ + \sum_{\substack{\Gamma \in \nas(R) \ : \\   \Gamma \cap \rf(R)\subset \hypb }} \deg_{R} \Gamma \minis .
        \end{equation}
        Now, by Favre and Rivera-Letelier's \cref{thmb3frl} we have that, since $L(R) >0$,
        \begin{equation*}
            (\deg R^n)^{-1}\sum_{\zeta \in \rf(R^n) \cap \PP^1(K)} \ \delta_\zeta \xrightarrow[ \ n\to\infty \ ]{} \ \nu_R \minis ,
        \end{equation*}
        which implies in particular that $\deg R^n -  \# (\rf(R) \cap \PP^1(K)) = o(\deg R^n)$.
        By \cref{formaditausuciascungamma} (a), all repelling fixed points of $R^n$ in $\hypb$ lie in some $\Gamma \in \nas(R^n)$ containing no type I repelling point. Using this fact and applying \cref{ildoppiogrado} and \cref{tipounotipo2} to the iterates of $R$, we find that 
        \begin{multline*}
            \sum_{\zeta \in \rf(R^n) \cap \hypb} \deg_{R^n} \zeta
            \ =  \sum_{\substack{\Gamma \in \nas(R^n) \ : \\   \Gamma \cap \rf(R^n) \subset \hypb }}   \sum_{\zeta \in \Gamma \cap \rf(R^n)} \deg_{R^n} \zeta   \\ 
            \leq \ 2 \left ( \sum_{\substack{\Gamma \in \nas(R^n) \ : \\   \Gamma \cap \rf(R^n) \subset \hypb }}  \deg_{R^n} \Gamma \right )
            = 2\left ( \deg R^n - \# \left (\rf  (R^n) \cap \PP^1(K) \right ) \right ) = o(\deg R^n)
        \end{multline*}
        Now it follows from the inequality above that the positive measure in the left-hand side of \cref{sparisconogliperbolici} has total mass tending to zero as $n\to \infty$, which implies our statement. 
    \end{proof}

\bibliographystyle{alpha}
\bibliography{bibbio}

%indirizzi
\bigskip
\bigskip

\noindent
\textsc{University of Washington,  Box 354350, Seattle, WA 98195, USA} \\
\textit{Email address}: \texttt{\href{loren20@uw.edu}{loren20@uw.edu}}

\end{document}